\documentclass[review,onefignum,onetabnum]{siamart220329}
\usepackage{amsmath,amssymb,amsfonts,enumerate,natbib,color,ifthen,hyperref}
\usepackage{xargs}
\usepackage[latin1]{inputenc}

\usepackage{float}
\usepackage[caption = false]{subfig}
\usepackage{graphicx}

\usepackage{aliascnt,bbm}
\def\nset{{\mathbb{N}}}
\def\rset{\mathbb R}
\def\zset{\mathbb Z}

\def\eqsp{\;}

\newcommand{\pscal}[2]{\left\langle#1,#2\right\rangle}

\newcommand{\un}{\ensuremath{\mathbbm{1}}}
\newcommand{\eqdef}{\ensuremath{\stackrel{\mathrm{def}}{=}}}
\newcommand{\eps}{\varepsilon}

\def\Xset{\mathsf{X}} 
\def\Zset{\mathcal{Z}} 
\def\F{\mathcal{F}} 
\def\B{\mathcal{B}} 
\def\e{\mathcal{E}}

\def\N{\mathcal{N}}

\def\D{\mathcal{D}}
\def\A{\mathcal{A}}

\def\H{\mathcal{H}}

\newcommandx\sequence[3][2=t,3=\zset]
{\ifthenelse{\equal{#3}{}}{\ensuremath{\{ #1_{#2}\}}}{\ensuremath{\{ #1_{#2}, \eqsp #2 \in #3 \}}}}

\def\PP{\mathbb{P}} 
\newcommand{\CPP}[3][]
{\ifthenelse{\equal{#1}{}}{{\mathbb P}\left(\left. #2 \, \right| #3 \right)}{{\mathbb P}_{#1}\left(\left. #2 \, \right | #3 \right)}}
\def\PE{\mathbb{E}} 
\newcommand{\CPE}[3][]
{\ifthenelse{\equal{#1}{}}{{\mathbb E}\left[\left. #2 \, \right| #3 \right]}{{\mathbb E}_{#1}\left[\left. #2 \, \right | #3 \right]}}

\def\W{\mathcal{W}}

\def\tv{\mathrm{tv}}
\def\Cset{\mathcal{C}} 

\newcommand{\normfro}[1]{\left\Vert#1\right\Vert_{\textsf{F}}}

\newcommand{\tnorm}[1]{\left\vert\!\left\vert\!\left\vert#1\right\vert\!\right\vert\!\right\vert}

\def\t{\textsf{T}}

\def\r{\textsf{r}}

\usepackage{bm}

\newtheorem{assumption}{H\hspace{-3pt}}

\newaliascnt{remark}{theorem}
\newtheorem{remark}[remark]{Remark}
\aliascntresetthe{remark}
\newaliascnt{example}{theorem}
\newtheorem{example}[example]{Example}
\aliascntresetthe{example}

\def\rmd{\mathrm{d}}

\def\1{\mathbbm{1}}

\DeclareMathOperator*{\argmin}{{\textsf{Argmin}}}

\usepackage[nothm]{yk}

\definecolor{ao(english)}{rgb}{0.0, 0.5, 0.0}

\begin{document}
\title{On the estimation rate of Bayesian PINN for inverse problems}

%

\author{Yi Sun\thanks{Department of Mathematics and Statistics, Boston University, Boston, MA (\email{ysun4@bu.edu}).}
\and Debarghya Mukherjee\thanks{Department of Mathematics and Statistics, Boston University, Boston, MA (\email{mdeb@bu.edu}).} 
\and Yves Atchad\'e\thanks{Department of Mathematics and Statistics, Boston University, Boston, MA (\email{atchade@bu.edu}).}}

\maketitle


\begin{abstract}
Solving partial differential equations (PDEs) and their inverse problems using Physics-informed neural networks (PINNs) is a rapidly growing approach in the physics and machine learning community. Although several architectures exist for PINNs that work remarkably in practice, our theoretical understanding of their performances is somewhat limited. In this work, we study the behavior of a Bayesian PINN estimator of the solution of a PDE from $n$ independent noisy measurement of the solution. We focus on a class of equations that are linear in their parameters (with unknown coefficients $\theta_\star$). We show that when the partial differential equation admits a classical solution (say $u_\star$), differentiable to order $\beta$,  the mean square error of the Bayesian posterior mean is at least of order $n^{-2\beta/(2\beta + d)}$.
Furthermore, we establish a convergence rate of the linear coefficients of $\theta_\star$ depending on the order of the underlying differential operator. 
Last but not least, our theoretical results are validated through extensive simulations. 
\end{abstract}

\begin{keywords}
partial differential equations, inverse problems, physics-informed neural network, nonparametric regression,  Bayesian deep learning, posterior contraction
\end{keywords}


\section{Introduction}\label{sec:intro}

We consider a class of inverse problems that consists in estimating the parameters of a partial differential equation from noisy measurements of the solution. Specifically, let $\Omega\subset \rset^{m}$ be a bounded open domain with a smooth boundary equipped with a probability measure $\nu$. For $j\geq 1$, let $L^2(\Omega,\rset^j,\nu)$ denote the $L^2$-space of $\rset^j$-valued functions on $\Omega$, i.e. 
$$
L^2(\Omega,\rset^j,\nu) \eqdef \left\{f: \Omega \mapsto \reals^j: \int_\Omega \|f(x)\|_2^2 \ \nu(\rmd x) < \infty \right\} \,.
$$
We consider a differential equation with parameter $\theta\in\rset^{d}$ that aims to find a smooth function $u\in L^2(\Omega,\rset,\nu)$ satisfying
\begin{equation}\label{main:pde}
\H_0 u + \theta^\t \H_1 u = f,\;\mbox{ with initial/boundary condition }\;\; \mathcal{B}u ={\bf 0},
\end{equation}
for some known function $f\in L^2(\Omega,\rset,\nu)$, and a pair of differential operators $\H_0:\; L^2(\Omega,\rset,\nu)\to L^2(\Omega,\rset,\nu)$ and $\H_1:\; L^2(\Omega,\rset,\nu)\to L^2(\Omega,\rset^{d},\nu)$. The operator $\H_0$ is typically (although not necessarily) the "time" partial derivative,  and the operator $\mathcal{B}:\;L^2(\Omega,\rset,\nu)\to L^2(\Omega,\rset^{d_1},\nu)$ imposes the boundary/initial conditions. A large class of differential equations (linear and nonlinear) that are linear in their parameters can be written in this form. 
A classical example of such an inverse problem arises from the heat equation as elaborated below:

\begin{example}[Heat equation]\label{ex:1}
Given $L>0,T>0$, and parameter $\theta\in\rset$, consider  the heat equation $u_t -\theta u_{xx}=0$ on $(0,L)$, with time domain $(0,T)$, initial condition $u(0,\cdot) = g(\cdot)$, and boundary condition $u(\cdot,0)=u(\cdot,L)=0$. Here $u_t$ (resp. $u_{xx}$) denotes the partial derivative of $u$ with respect to $t$ (resp. the second order partial derivative of $u$ with respect to $x$). With $\Omega = (0,T)\times (0,L)$, this equation can be written as $\H_0 u -\theta \H_1 u = f$, where $\H_0 u = u_t$, $\H_1 u = u_{xx}$, and $f=0$. The boundary/initial condition operator is given by
\[\mathcal{B} u =\left(\begin{array}{c}u_{\{t=0\}}-g\\ u_{\{x=0\}}\\ u_{\{x=L\}}\end{array}\right).\]
 where $u_{\{t=0\}}$ is the map $u_{\{t=0\}}(x)= u(0,x)$, with similar definition for $u_{\{x=0\}}$ and $u_{\{x=L\}}$.   Here $d=1$, $d_1=3$.
\end{example}

Let $(\theta_\star, u_\star)$ be a tuple that satisfies the PDE \eqref{main:pde}. 
We assume to observe some noisy measurements of $u_\star$ as noted below: 
\begin{assumption}\label{H1}
We have $n$ i.i.d. random locations and observations $({\bf s}_i,Y_i)\in\Omega\times \rset$, where
\[Y_i \; \vert \;  {\bf s}_i \stackrel{ind.}{\sim} \mathbf{N}\left(u_{\star}({\bf s}_i),\sigma^2\right),\]
where $u_\star:\Omega\to\rset$ is the unique solution of the pde (\ref{main:pde}) with $\theta=\theta_\star$, for some unknown vector $\theta_\star\in\rset^d$. Here, $\textbf{N}(m,v^2)$ denotes the univariate Gaussian distribution with mean $m$ and variance $v^2$. The variance parameter $\sigma^2$ is assumed to be known. Throughout we write $\PP$ for the joint distribution of $({\bf s}_1, Y_1),\ldots,({\bf s}_n, Y_n)$, and $\PP_n$ for their corresponding empirical measure.
\end{assumption}

Our goal is to estimate $(\theta_\star, u_\star)$ using the noisy measurements. Over the last few years, physics-informed neural network (PINN) has taken the numerical pde literature by storm (\cite{raissi:etal:19,yu2018deep,sirignano2018dgm,sun2020surrogate}, and \cite{cuomo:etal:22} for an extensive review of the literature). In the setting of H\ref{H1}, the approach consists of estimating $u_\star$ by regression while explicitly using the information that the true function $u_\star$ is the solution of a PDE.  In this paper, as typically done in the PINN literature, we resort to a sieve-based approach for estimating $u_\star$, i.e., we consider a sequence of increasingly complex parametric models. Hence let $\F\eqdef \{u_W,\;W\in\rset^q\}$ be a function class, where $u_W:\;\Omega\to\rset$ is a function with parameter $W\in\rset^q$. Due to their superior empirical performances, $u_W$ is typically taken as a deep neural network, and we follow that practice, although our results can be applied more broadly. Given the data $\D\eqdef\{({\bf s}_i,Y_i),\;1\leq i\leq n\}$, PINN in its frequentist formulation estimates jointly  $\theta$ and $W$ by minimizing the loss function 
\begin{equation}\label{loss}
\frac{1}{2\sigma^2}\sum_{i=1}^n\left(Y_i - u_W({\bf s}_i)\right)^2 + \frac{\lambda}{2} \left\{\alpha_1|f - \H_0 u_W -  \theta^\t\H_1 u_W|_2^2 +\alpha_2|\mathcal{B} u_W|_2^2\right\},
\end{equation}
for some regularization parameters $\lambda\geq 0$, and $\alpha_1,\alpha_2\geq 0$, where $|\cdot|_2$ denotes the $L^2$-norm on the appropriate function space. 
In this work, we approach the problem from a Bayesian perspective. Hence, starting from a standard Gaussian prior for $\theta\in\rset^d$, and a sparsity inducing prior density $\Pi_0$ for $W\in\rset^q$ (see Section \ref{sec:prior} below for our choice of $\Pi_0$), we define the informative PINN prior distribution for $(\theta,W)$ as the probability measure  on $\rset^d\times\rset^q$  with density proportional to 
\begin{equation}\label{pinn:prior}
(\theta,W)\mapsto \Pi_0(W)\exp\left(-\frac{\lambda}{2} \left\{\alpha_1|f - \H_0 u_W -  \theta^\t\H_1 u_W|_2^2 +\alpha_2|\mathcal{B} u_W|_2^2\right\} -\frac{1}{2}\|\theta\|_2^2\right).
\end{equation}
For appropriately large choices of $\lambda$, any realization $(\theta,W)$ from the  PINN prior distribution (\ref{pinn:prior}) produces $(\theta,u_W)$ that is biased toward solving the pde (\ref{main:pde}). Rigorous general results (not specific to PINN) of this flavor can be found, for instance, in \cite{hwang:80}. When $\theta$ is known, realizations from the conditional distribution of $W$ given $\theta$ in (\ref{pinn:prior}) produce $u_W$ that approximately solve the pde (\ref{main:pde}) for the given $\theta$. This corresponds to the initial PINN methodology of  \cite{raissi:etal:19}. 

For some convenience in the analysis, we will use the prior (\ref{pinn:prior}) with $\alpha_1=1$, and $\alpha_2 = 0$. Since the parameter $\theta$ does not appear in the boundary condition, the case $\alpha_2>0$ (although useful in practice) does not induce any fundamentally new behavior, and at the expense of more involved notations, our analysis can be easily modified to handle $\alpha_2>0$. Given the observed data $\D$, the data generating model postulated in (H\ref{H1}), and the PINN prior distribution (\ref{pinn:prior}), we thus consider the posterior distribution
\begin{multline}\label{post:joint:ideal}
\Pi(\rmd\theta,\rmd W\vert\D) \propto \\
\exp\left(-\frac{1}{2\sigma^2}\sum_{i=1}^n\left(Y_i - u_W({\bf s}_i)\right)^2- \frac{\lambda}{2} |f - \H_0 u_W -  \theta^\t\H_1 u_W|_2^2 -\frac{\|\theta\|_2^2}{2}\right) \Pi_0(\rmd W) \rmd\theta.
\end{multline}

\subsection{Main contributions}
We study the behavior of the posterior distribution (\ref{post:joint:ideal}) as $n\to\infty$ (and $q\to\infty$). We consider the case where the operator $\H \eqdef (\H_0,\H_1)$ is a differential operator that involves derivatives up to order $\tau$, for some $\tau>0$. Under some additional regularity conditions, we establish in Theorem \ref{thm:marg:theta} that when the true pde solution $u_\star$ is a classical solution that possesses derivatives up to order $\beta\geq \tau$, the mean square error of the posterior mean of $\theta$ under (\ref{post:joint:ideal}) satisfies (up to log terms that we ignore),
\begin{equation}
\label{eq:theta_posterior_conv}
\PE\left[\left\|\int_{\rset^d} \vartheta \Pi^{(\theta)}(\rmd \vartheta\vert\D) -\theta_\star \right\|_2^2\right] \lesssim n^{-2(\beta-\tau)/(m + 2\beta)},
\end{equation}
where $\Pi^{(\theta)}$ is the $\theta$-marginal of (\ref{post:joint:ideal}). In fact we obtain this rate by showing that the 2-Wasserstein distance between $\Pi^{(\theta)}$ (the posterior distribution of $\theta$) and the Gaussian distribution $\mathbf{N}(\theta_\star,\lambda^{-1}\Sigma_\star)$ converges to zero at the rate given above (see Section \ref{sec:main:res} for the definition of $\Sigma_\star$). Our simulation results suggest that, in general, this convergence does not hold in the total variation metric. We also show in Theorem \ref{thm:margW} that the pde solution  $u_\star$ is estimated at a rate at least as fast as the nonparametric minimax optimal rate:
\[\PE\left[\int_{\rset^q}\;|u_W - u_\star|_2^2\Pi^{(W)}(\rmd W\vert \D)\right] \;\lesssim n^{-2\beta/(m + 2\beta)},\]
where $\Pi^{(W)}$ denotes the $W$-marginal of (\ref{post:joint:ideal}). We conjecture that the PINN posterior distribution actually converges faster than the nonparametric minimax rate. Indeed, in a limiting case of an infinitely strong PINN prior, we construct an estimator which, although computationally intractable in general, achieves the parametric rate $n^{-1/2}$ (Theorem \ref{thm:main_param}). 

\subsection{Related work}
Despite their popularity, the theoretical aspects of PINN-related methods remain under-studied. 
Most of the existing theoretical literature on PINN deals with the forward problem of estimating $u$ for a given $\theta$, when the $L^2$-norm $|\cdot|_2$ is replaced by a Monte Carlo approximation (\cite{mishra:etal:22,shin:etal:2023,lu-chen:etal:2022}). The objective in this line of work is to derive the mean square error of the PINN solution as a function of the Monte Carlo sample size. More related to our work, \cite{mishra:etal:21} considers a PINN inverse problem where the parameter of interest appears in the boundary condition. They work mainly in a noiseless (or small-magnitude noise) data regime. Furthermore, they assume a continuous space observation model, which is substantially different from the framework considered here.

Inverse problems as related to differential and partial differential equations have a long history that predates the current PINN literature (\cite{ramsey:etal:07,xun:Etal:2013,mcgoff:etal:15}). In particular, we note the close similarity between the PINN methodology and the approach by \cite{xun:Etal:2013}. The estimation rate in pde-driven inverse problems has been extensively studied in statistics in recent years (\cite{spokoiny2020bayesian,giordan:etal:2020,nickl:etal:2020,monard:etal:21,dehoop:etal:23}). However, none of these results are directly applicable to our setting or PINN more generally, since in this literature, the forward map is typically assumed known and is not parametrized as in PINN. 

A Bayesian version of PINN was proposed and studied in (\cite{yang:bpinn:21}) from an empirical viewpoint, where the authors noted a certain robustness of Bayesian PINN to measurement noise. However, no theoretical analysis is proposed.

A related version of the inverse problem considered in this work assumes that $\theta$ is high-dimensional and sparse (\cite{deepMod:21,Chen2020Physicsinformed}). The goal is then to estimate $\theta$ under a sparsity assumption. These physics discovery problems can be viewed as high-dimensional versions of the problem investigated here. Thus, our work serves as a stepping-stone toward a theoretical understanding of high-dimensional physics discovery inverse problems.


\subsection{Outline of the paper}
We close the introduction with a description of the function class $\F=\{u_W,\;W\in\rset^q\}$, the prior distribution $\Pi_0$  for $W$, and a description of the main notations that we use throughout the paper. Section \ref{sec:main:res} deals with the estimation rate for $u_\star$. A Bernstein-von Mises theorem for the marginal distribution of $\theta$ is discussed in Section  \ref{sec:bvm}. Section \ref{sec:illust} presents our numerical illustrations of the theoretical findings.

\subsection{The function class $\F=\{u_W\}$}\label{sec:Fw}
Although our results apply widely, we follow the PINN literature and focus on the case where the function class $\F$ is constructed using neural networks. Let $\xi>0$ be the depth of the network. Let $(p_\xi,\ldots,p_{0})$ be a sequence of integers representing the sizes of the layers of the network, with $p_0=m$, and $p_\xi=1$. For $1\leq i\leq \xi$, let $\mathsf{A}_i:\;\rset^{p_i}\to \rset^{p_i}$ be a component-wise application of a 1-Lipschitz function $\mathsf{a}_i:\rset\to\rset$. For  $B\in\rset^{p_i\times p_{i-1}}$, and $b\in\rset^{p_i}$, we set
\begin{equation}\label{psi:fun}
\Psi_{B,b}^{(i)}({\bf z}) \eqdef \mathsf{A}_i(B {\bf z}+b),\;\;\;{\bf z}\in\rset^{p_{i-1}}.\end{equation}
We consider functions $u_W$ of the form 
\begin{equation}\label{u:fun}
u_W({\bf x}) = \Psi_{W_{\xi},w_\xi}^{(\xi)}\circ \cdots  \circ \Psi_{W_1,w_1}^{(1)} ({\bf x}),\;\;\;{\bf x}\in\rset^{m},
\end{equation}
with parameter $W= (W_\xi,w_\xi,\ldots,W_1,w_1)$, where $W_i\in\rset^{p_i\times p_{i-1}}$, and $w_i\in\rset^{p_i}$, and where $f\circ g$ is the composition of $f$ with $g$. For convenience, and by vectorization we view $W$ as an element of $\rset^q$, where 
\[q \eqdef \sum_{i=1}^\xi p_i(1+p_{i-1}). \]
We assume that the activation functions $\mathsf{a}_\ell$ are of class $\Cset^\infty$, which implies that $u_W$ is of class $\Cset^\infty$. The activation function $\mathsf{a}_\xi$ in the last layer is typically taken as the identity function. 

\subsection{The prior distribution $\Pi_0$}\label{sec:prior}
We use sparsity to control the complexity of the function class $\{u_W,\;W\in\rset^q\}$. We use a version of the spike-and-slab prior distribution taken from \cite{AB:19}. We set $\mathcal{S} \eqdef \{0,1\}^{q}$. To construct the prior $\Pi_0$, we first define a pair of random vectors $(\Lambda,W)\in\mathcal{S}\times\rset^q$ as follows. Let $\Lambda_j\stackrel{i.i.d}{\sim}\textbf{Ber}((1+q^{\mathsf{u}+1})^{-1})$, $1\leq j\leq q$, for some sparsity parameter $\mathsf{u}\geq 1$, where $\mathbf{Ber}(\alpha)$ denotes the Bernoulli distribution with parameter $\alpha\in (0,1)$.  Given $\Lambda\in \mathcal{S}$, the components of $W$ are independent with joint density on $\rset^q$ given by
\[W\mapsto \prod_{j:\;\Lambda_{j}=1} \;\sqrt{\frac{1}{2\pi}} e^{-\frac{1}{2} W^2_{j}}  \prod_{j:\;\Lambda_{j}=0} \sqrt{\frac{\rho_0}{2\pi}}e^{-\frac{\rho_0}{2}W_{j}^2},\]
for some given parameter $\rho_0>1$. We denote $\bar\Pi_0$ the joint distribution of $(\Lambda,W)$, and we let $\Pi_0$ (our prior on $\rset^q$) be the distribution of $ W \odot \Lambda $, where for $a,b\in\rset^q$, $a\odot b$ denotes the component-wise product of $ a $ and $ b $. By construction, for any measurable function $f:\rset^q\to\rset$, we have
\[\Pi_0(f) \eqdef \int_{\rset^q} f(v)\Pi_0(\rmd v) = \int_{\mathcal{S}\times \rset^q} f(\Lambda\odot W)\bar\Pi_0(\rmd \Lambda,\rmd W).\]

\subsection{Notation}\label{sec:notations}
For $\alpha\geq 1$, $\|\cdot\|_\alpha$ denotes the $\ell^\alpha$-norm on finite-dimensional Euclidean spaces ($\rset^d$, $\rset^n$, $\rset^{d}$, etc, -- which space should be clear from the context), and we use $a^\t b$ to denote the inner product between two finite dimensional vector $a,b$. As usual, we allow $\alpha=+\infty$ (resp. $\alpha=0$),  by defining $\|a\|_\infty\eqdef \max_{1\leq i\leq d} |a_i|$ (resp. $|a|_0$ is the number of non-zero components of $a$). Given a matrix $A$, $\lambda_{\textsf{min}}(A)$ (resp. $\|A\|_{\textsf{op}}$) denotes its smallest (resp. largest) singular value.

Let $\Omega\subset \rset^{m}$ be as above with a probability measure $\nu$ (typically the uniform measure on $\Omega$).  Given an integer $j\geq 1$, $L^2(\Omega,\rset^j,\nu)$ denotes the Hilbert $L^2$-space of $\rset^j$-valued functions with inner product 
\[\pscal{f_1}{f_2}\eqdef \int_{\Omega} f_1({\bf x})^\t f_2({\bf x})\nu(\rmd{\bf x}), \;\; \mbox{ and norm } \;\; |f|_2\eqdef \sqrt{\pscal{f}{f}}.\] 
For $f:\;\Omega\to\rset^j$, we set $|f|_\infty\eqdef \sup_{x\in\Omega}\|f(x)\|_\infty$.

We use multi-index derivatives: given ${\bf k}=(k_1,\ldots,k_m)$, where $k_i\geq 0$ is an integer,
\[D^{\bf k} u(x) \eqdef \frac{\partial^{|{\bf k}|}u(x)}{\partial x_1^{k_1}\ldots \partial x_m^{k_m}},\]
where $|{\bf k}| =\sum_i k_i$. Given $\beta>1$, we set 
\[|u|_{\Cset^\beta}\eqdef \sum_{{\bf \alpha}:\;|{\bf \alpha}|\leq \lfloor \beta\rfloor}\; |D^{\bf \alpha} u|_\infty + \sum_{\alpha:\;|\alpha| = \lfloor \beta\rfloor}\; \sup_{x\neq y}\frac{|D^{\bf \alpha} u(y) - D^{\bf \alpha} u(x)|}{\|y-x\|_2^{\beta- \lfloor \beta\rfloor}},\]
where $\lfloor \beta\rfloor$ is the largest integer strictly smaller than $\beta$. $|u|_{\Cset^\beta}$ is the Holder norm of $u$. We set $\Cset^{\beta}(\Omega)$  (resp. $\Cset^{\beta}(\Omega,b)$) as the set of all function $u:\;\Omega\to\rset$ with finite Holder norm (resp. with Holder norm bounded by $b$). 


Throughout the paper, we use $C_0,C_1,C_2,\etc$ to denote absolute constants and $C$ to denote a generic absolute constant that depends on $C_0,C_1,\etc$. The actual value of $C$ may change from one appearance to the next. Similarly, $c_0,c_1,\etc$ denote problem-dependent constants (that do not depend on the sample size $n$ or the model size $q$). Specifically, these constants typically depend on the noise parameter $\sigma^2$, the dimension $d$, and the true parameter $\theta_\star$. We will also use $c$ to denote a generic constant that depends on $c_0,c_1,\etc$, and that we do not track. The actual value of $c$ may change from one appearance to the next.

\section{Estimation rate of $u_\star$}\label{sec:main:res}
We  maintain throughout the basic assumption that the operator $\H = (\H_0,\H_1)$  is a differential operator of order $\tau>0$,  in the sense that for all $\beta>\tau$, and for $u,v\in\Cset^\beta(\Omega)$,
\begin{equation}
\label{lip:H}
\left|\H u - \H v\right|_2 \leq C_0 \max_{{\bf k}:\;|{\bf k}|\leq \tau} \;\left|D^{\bf k} u - D^{\bf k} v\right|_\infty,
\end{equation}
for some absolute constant $C_0$. 

Because the pde equation is linear in $\theta$, and the log-prior is quadratic in $\theta$, it is straightforward to integrate out $\theta$ from the posterior distribution (\ref{post:joint:ideal}). We first introduce some appropriate notations to do this. For $W\in\rset^q$, we define $\bar f_W \eqdef f - \H_0 u_W$. Let $\Phi_W\in\rset^d$ be the vector with $j$-th component given by
\[(\Phi_W)_j \eqdef \pscal{\bar f_W}{(\H_1 u_W)_j},\] 
and where $(\H_1 u_W)_j$ denotes the $j$-component of $\H_1 u_W$. We define $\Phi_\star\in\rset^{d}$ similarly, replacing $u_W$ by $u_\star$.  Let $\Sigma_W\in\rset^{d\times d}$ given by
\[(\Sigma_W)_{j,k}  \eqdef   \pscal{(\H_1 u_W)_j}{(\H_1 u_W)_k}\;  + \;\frac{1}{\lambda}\textbf{1}_{\{j=k\}},\;\; 1\leq j,k\leq d.\]
 We define $\Sigma_\star$ similarly by replacing $u_W$ by $u_\star$. Given $u\in L^2(\Omega,\rset,\nu)$ we define
\begin{equation}\label{def:J}
{\cal J}(u)\eqdef \min_{\theta\in\rset^d}\;\left[|f - \H_0 u -\theta^\t \H_1 u|_2^2 +\frac{1}{\lambda}\|\theta\|_2^2\right].\end{equation}
It is easy to check that ${\cal J}(u_W) = |\bar f_W|_2^2 - \Phi_W^\t \Sigma_W^{-1}\Phi_W$, and we can integrate out $\theta$ from the posterior distribution (\ref{post:joint:ideal}) to obtain the marginal distribution of $W$ given by 
\begin{equation}\label{marginalW}
\Pi(W\vert\D) \propto \frac{\Pi_0(W)}{\sqrt{\det(\Sigma_W)}}\exp\left(-\frac{1}{2\sigma^2}\sum_{i=1}^n(Y_i - u_W({\bf s}_i))^2 -\frac{\lambda}{2}{\cal J}(u_W)\right).
\end{equation}
At times, when the distinction is needed we will write $\Pi^{(W)}$ to denote (\ref{marginalW}).
We make the following basic assumption on the function class $\F\eqdef \{u_W,\;W\in\rset^q\}$.
\medskip

\begin{assumption}\label{H:lip}
\begin{enumerate}
\item For all $W_0\in\rset^q$, and $r>0$, we can find $L_{W_0,r}\geq 1$ such that if  $W_1,W_2\in\rset^q$ satisfy $\max(\|W_1 - W_0\|_2, \|W_2 - W_0\|_2) \leq r$, then it holds
\[\left|u_{W_1} - u_{W_2}\right|_\infty \leq L_{W_0,r}\;\|W_1 - W_2\|_2.\]
\item For all $W\in\rset^q$, $u_W$  has derivatives to the order $\tau$, and there exists $\kappa>0$ and a constant $c_1$ such that for all $W_1,W_2\in\rset^q$, such that $|u_{W_1} - u_{W_2}|_\infty\leq \varepsilon$,  we have
\[\max_{{\bf k}:\;|{\bf k}|\leq \tau} \;\left|D^{\bf k} u_{W_1} - D^{\bf k} u_{W_2}\right|_\infty \leq c_1 \varepsilon^{\kappa}.\]
\end{enumerate}
\end{assumption}

\medskip

\begin{remark}
The local Lipschitz condition imposed in H\ref{H:lip}-(1) is known to hold for most deep neural net architectures with Lipschitz activation functions (see e.g. Proposition 6 of \cite{taheri:etal:21}). Specifically, if $u_W$ is a deep neural network as constructed in Section \ref{sec:Fw} with 1-Lipschitz activation functions, then with $W_1,W_2\in\rset^q$ satisfying $\max(\|W_1 - W_0\|_2, \|W_2 - W_0\|_2) \leq r$, Proposition 6 of \cite{taheri:etal:21} and some easy calculations yield
\[|u_{W_1} - u_{W_2}|_\infty \leq L_{W_0,r} \|W_1 - W_2\|_2,\]
with
\begin{equation}\label{LWr}
L_{W_0,r}  = \left(\sup_{{\bf x}\in\Omega}\|{\bf x}\|_2\right) \; \sqrt{\xi}\; \left( 1 + \frac{(\|W_0\|_2 + r)^2}{\xi}\right)^{\xi},\end{equation}
where $\xi$ is the depth of the DL function.
\end{remark}

\begin{remark}
Since we aim to solve pdes, it is natural to require the function $u_W$ to be smooth. H\ref{H:lip}-(2) further requires a Holder-type assumption on the differential operator $D^{{\bf k}}$ over the function class $\F$. This assumption is less common in the literature, and we check its validity in Theorem \ref{thm:approx:deriv}. 
\end{remark}

\medskip

Let $\epsilon:\;\nset\to\rset$ be a nonincreasing function such that for all $s\geq 1$
\begin{equation}\label{def:eps}
\min\left\{|u_W - u_\star|_\infty,\;W\in\rset^q:\;\|W\|_0\leq s \right\} \leq \epsilon(s).\end{equation}
$\epsilon(s)$ defines the approximation skills of the function class $\{u_W\}$ at the sparsity level $s$, and is a nonincreasing function of $s$. Of importance are sparsity levels $s>1$ at which the approximation error $\epsilon(s)$ matches the statistical error:
\begin{equation}
\label{eq:approx_s_val}
\epsilon(s) \leq \sigma\sqrt{\frac{\mathsf{u} s\log(q)}{n}}.
\end{equation}

We make the following assumption.
\medskip

\begin{assumption}\label{H:approx}
There exists $\epsilon:\;\nset\to\rset$ such that (\ref{def:eps}) holds. Given $n$ and $q$, let $s_0\geq 2$ be the smallest integer satisfying (\ref{eq:approx_s_val}), and $\epsilon_0 \eqdef \epsilon(s_0)$.  We assume that there exists $W_0\in\rset^q$, with $\|W_0\|_0\leq s_0$ that satisfies $|u_{W_0} - u_\star|_\infty\leq \epsilon_0$, and
\begin{equation}\label{tech:cond:2}
\min\left(\lambda_{\textsf{min}}(\Sigma_{W_0}),\lambda_{\textsf{min}}(\Sigma_{\star})\right)\geq C_2, \end{equation}
for some absolute constant $C_2$. Furthermore, we assume that $\lim_{n\to\infty}\epsilon_0 = 0$, and 
\begin{equation}\label{tech:cond:3}
(\|W_0\|_\infty +1)^2 + \log\left(Ls_0^{1/2} / \epsilon_0\right) \leq C_3 \mathsf{u}\log(q),
\end{equation}
for some absolute constant $C_3$, where  $L\eqdef L_{W_0,1}$ is as in H\ref{H:lip}.
\end{assumption}
\medskip

\begin{remark}\label{rem:H3}
We recall that we aim to estimate $u_\star$ by the sieve method, i.e., we approximate $u_\star$ by a collection of functions $\cF \equiv \cF_n$, where the \emph{size} of the function class typically grows with the sample size $n$. The assumption that $\epsilon_0\to 0$ as $n\to\infty$ in H\ref{H:approx} highlights the relationship between the sample size $n$ and the model size: given $n$, we assume that the model $\cF$ is chosen such that its approximation error at some sparsity level $s$ can match (or be smaller than) the statistical error at sparsity level $s$.

There has been a flurry of research activities in recent years to derive precise estimates on $\epsilon(s)$ for various DL function classes, using various norms (\cite{yarostky:17,shieber:20,lu:shen:etal:21,belomestny:2022}). 
For example for  piecewise polynomial  activation functions, and $u_W$ as defined in Section \ref{sec:Fw}, it is known (see e.g. \cite{yarotsky:20,lu:shen:etal:21}) that at the sparsity level $s$, and depth $\xi=O(\log(s))$, $\F$ can approximate a function $u_\star\in\Cset^\beta(\Omega)$ with precision $(1/s)^{\beta/m}$. 
This yields 
\[\epsilon(s) \lesssim  (1/s)^{\beta/m}. \]
In that case, solving for $s_0$ in (\ref{eq:approx_s_val}) yields
\[ s_0 \sim \left(\frac{n}{\log(q)}\right)^{\frac{m}{m + 2\beta}},\;\mbox{ and }\;\; \epsilon_0 \sim  \left(\frac{\log(q)}{n}\right)^{\frac{\beta}{m+2\beta}}.\]

In the high-dimensional setting considered here where $q$ is typically much larger than $n$, (\ref{tech:cond:3}) is typically satisfied. 
\end{remark}

\medskip

Ultimately, we need some assumptions on the stability of the pde operator $\H$ as presented below: 
\medskip

\begin{assumption}\label{H:stability}
With $\kappa$ as in H\ref{H:lip}-(2), there exists an absolute constant $C_4$ such that for all $\varepsilon>0$, if $|u_W-u_\star|_2\leq \varepsilon$, then $|\H u_W - \H u_\star|_2 \leq C_4 \varepsilon^\kappa$.
\end{assumption}

\medskip

\begin{remark}
H\ref{H:stability} is similar to H\ref{H:lip}-(2) and both can  be checked in a similar way if $u_\star$ is a classical solution. Indeed, if the pde solution $u_\star$ is a classical solution, and $u_\star\in\Cset^\beta(\Omega)$ for some $\beta>\tau$, then by (\ref{lip:H}), and Theorem \ref{thm:approx:deriv} below, H\ref{H:stability} holds with  $\kappa = 1-\tau/\beta$.
\end{remark}

\medskip

Given $s\geq 0$, and constant $C$, we set $\tau_s\eqdef \sqrt{(2+\mathsf{u})(1+s)\log(q)}$, and
\[\F_{s}=\F_{s,C}\eqdef\left\{u_W\in\F:\; \|W\|_0\leq s,\;\|W\|_\infty\leq \tau_s,\;\;{\cal J}(u_W) \leq \frac{C n\epsilon_0^2}{\lambda\sigma^2} \right\},\] 
where ${\cal J}$ is as defined in (\ref{def:J}). We show below that the PINN posterior distribution puts high probabilities on the sets $\F_s$. Specifically, the following theorem is proved as Lemma \ref{lem:good:set} in Section \ref{sec:prior:concentration}.

\begin{theorem}\label{thm:prio:conc}
Assume H\ref{H1}-H\ref{H:stability}. Let Let $\epsilon_0$, $W_0$ be as in H\ref{H:approx}. We can find an absolute constant $C$ such that with $s = C \|W_0\|_0$, it holds,
\[\PE\left[\Pi\left(\F_{s,C}\vert \D\right)\right] \geq 1 - C_0 e^{-\frac{n\epsilon_0^2}{2\sigma^2}},\]
for some absolute constant $C_0$. 
\end{theorem}

To highlight this point, we note that if a function $u\in L^2(\Omega,\rset,\nu)$ solves the pde equation  (\ref{main:pde}) (without the initial/boundary condition) for some parameter $\theta$, say, then
\[{\cal J}(u) = \frac{1}{\lambda}\theta^\t\left(\un - \frac{1}{\lambda}\Sigma_u^{-1}\right)\theta\leq \frac{\|\theta\|_2^2}{\lambda}.\]
Hence with the size of $\theta$ remaining bounded, and for $\lambda$ large, functions that approximately solve the pde equation (\ref{main:pde}) satisfies ${\cal J}(u) \lesssim C/\lambda$ for some constant $C$. Theorem \ref{thm:prio:conc} thus implies that for $\lambda$ large, the PINN posterior distribution inherits the inductive bias of the prior and restricts the search of a solution for the nonparametric regression problem on approximate solutions of the pde (\ref{main:pde}). The rate of convergence of PINN, therefore, depends mainly on the complexity of the function class $\F_s$, as measured, for instance, by its covering number. The next assumption models the metric entropy of $\F_s$.

\medskip

\begin{assumption}\label{H:entropy}
Given $s\geq 2$, let $3\leq b<\infty$ be such that
\begin{equation}\label{boundedness}
\sup_{u\in\F_s}\;|u - u_\star|_\infty \leq b.
\end{equation}
There exists $V_1=V_1(s)$ and $V_2 = V_2(s)\geq 6b$ such that for all $0<\varepsilon\leq \tau_s$,
\begin{equation}\label{entropy}
\log\N(\varepsilon,\F_{s},\|\cdot\|_\infty)\leq V_1 \log\left(\frac{V_2}{\varepsilon}\right),
\end{equation}
where $\N(\varepsilon,\F_{s},\|\cdot\|_\infty)$ denotes the $\varepsilon$-covering number of $\F_{s}$ in the $L_\infty$ norm. \end{assumption}

\medskip

%

The following result establishes the rate of convergence of the PINN posterior and is proved in Section \ref{sec:proof:thm:margW}. We assume that the prior parameter $\lambda>0$ is taken such that
\begin{equation}\label{cond:lambda}
\|\theta_\star\|_2^2 \lambda \leq C_5 n\epsilon_0^{2(1-\kappa)}
\end{equation}
for some absolute constant $C_5$.  We also impose the following mild technical condition: there exists an absolute constant $C_6$ such that
\begin{equation}\label{tech:cond:1}
d\left(\log\left(n\|\Sigma_\star\|_{\textsf{op}}\right) + \|\theta_\star\|_\infty^2\right) \leq  \frac{C_6 n\epsilon_0^2}{2\sigma^2}.
\end{equation}

\medskip

\begin{theorem}\label{thm:margW}
Assume H\ref{H1}-H\ref{H:entropy}, (\ref{cond:lambda}), and (\ref{tech:cond:1}).  Then  we can  find absolute constants $C$ and $M\geq 1$ such that with
\begin{equation}
\label{choice:s_1}
 s = C s_0,\;\mbox{ and } \r \eqdef 2(b+\sigma) \sqrt{\frac{V_1\log(V_2\sqrt{n})}{n}},
 \end{equation}
where $V_1=V_1(s)$, $V_2=V_2(s)$ are as defined in H\ref{H:entropy}, it holds for all $n$ large enough,
\[\PE\left[\Pi\left(\left\{W\in\rset^q:\;|u_W-u_\star|_2 > M \r\right\}\;\vert\;\D\right)\right] \leq c \left(e^{-n\epsilon_0^2/(2\sigma^2)}  + e^{-V_1\log(V_2\sqrt{n})/C}\right),\]
where $c$ is a constant that depends only on $c_1,c_2$ and $\sigma^2$.
\end{theorem}

\medskip

\subsection{Minimax nonparametric rate}
Bounding the metric entropy of the function class $\F_s$ in a way that leverages the pde information ${\cal J}(u_W) \leq C n\epsilon_0^2 /(\lambda\sigma^2)$ has proved challenging. However, we note that $\F_{s}\subseteq\{u_W\in\F:\; \|W\|_0\leq s,\;\|W\|_\infty\leq \tau_s\}$, and the metric entropy of this latter set is straightforward to bound using H\ref{H:lip}-(1). See also (\cite{shieber:20,lu:shen:etal:21}). Indeed, for all $W,W'$ in the ball $\{W\in\rset^q:\;\|W\|_0\leq s,\; \|W\|_\infty\leq \tau_s\}$, $\max(\|W\|_2,\|W'\|_2)\leq s^{1/2}\tau_s$. Hence by H\ref{H:lip}-(1), $|u_W - u_{W'}|_\infty \leq L s^{1/2} \|W - W'\|_\infty$, where $L = L_{{\bf 0},s^{1/2}\tau_s}$. The size of all the $(\varepsilon'=\varepsilon/L s^{1/2})$-cover of all the $\tau_s$-balls of all the $s$-sparse subspaces of $\rset^q$ in the infinity-norm is bounded from above by
\[{q\choose s} \left(1 + \frac{2r_s}{\varepsilon'}\right)^s.\]
It follows that 
\[\log \N(\varepsilon,\F_{s},\|\cdot\|_\infty) \leq s \log(q) + s\log\left(1 + \frac{2s^{1/2}\tau_s L}{\varepsilon}\right),\]
and H\ref{H:entropy} holds with
\[V_1(s) = s,\;\;\mbox{ and }\;\; V_2(s) = qr_s\left(1 + 2s^{1/2} L\right) \leq c\sqrt{\xi} q\left( 1 + s^2\log(q)\right)^{\xi+1/2},\]
for some constant $c$ that depends only on $\mathsf{u}$, and $\sup_{x\in\Omega}\|{\bf x}\|_2$, where $\xi$ denotes the depth of the neural network class. Hence, with $s= C s_0$,
\[V_1(s) = s = C s_0,\;\mbox{ and }\;\;\log(\sqrt{n}V_2(s)) \lesssim \log(n) + \xi\log(q),\]
where $\xi$ is the depth of the neural network. Furthermore, we have seen in Remark \ref{rem:H3} that if $u_\star$ is a classical solution and $u_\star\in\Cset^\beta(\Omega)$, then for piecewise polynomial activation at depth $\xi=O(\log(p))$, we have
\[s_0 \sim \left(\frac{n}{\log(q)}\right)^{\frac{m}{m + 2\beta}}.\]
As a result, up to log terms that we ignore, we obtain the estimation rate
\[\r \lesssim \sqrt{\frac{s_0}{n}} \lesssim \left(\frac{1}{n}\right)^{\frac{\beta}{m+2\beta}}.\]
Hence, Bayesian PINN estimates $u_\star$ at least at the nonparametric minimax rate $n^{-\beta/(m + 2\beta)}$, that is, at the optimal rate achievable when $u_\star$ is known only to be $\beta$-smooth (\cite{stone:82}) without any PDE constraint.  We note that our metric entropy calculations of the function class $\F_s$ are rather crude, as they ignore the inductive bias induced by the PINN prior. Thus, the PINN posterior distributions likely enjoy a faster convergence rate than the abovementioned nonparametric minimax rate. To understand this, consider the limiting case when $\lambda = +\infty$, i.e., the prior put all the mass on $\{f \in L^2(\Omega,\rset,\nu):\; {\cal J}(u)=0\}$, that is, the set of functions that are solutions of the pde equation $\H_0 u +\theta^\t\H_1 u =f$ for some $\theta\in\rset^d$. In the subsequent section, we show that estimating $u_\star$ at a parametric rate (up to a log factor) is possible in this limiting case.

\subsection{Estimation at the parametric rate}
\label{sec:est_param_rate}
We assume that $\theta_\star$ belongs to a compact set $\Theta\subset\rset^d$, and for each $\theta\in\Theta$ the pde equation (\ref{main:pde}) (including the initial/boundary condition) admits a unique solution  $u_\theta\in\Cset^\beta(\Omega)$ for some given and fixed $\beta$. Furthermore we assume  that the map $\Gamma:\;\Theta\to\Cset^\beta(\Omega)$ that maps $\theta\to u_\theta$ is Lipschitz. Specifically,
\begin{assumption}
\label{assm:gamma_smooth}
There exists a constant $L > 0$ such that: 
$$
|\Gamma(\theta_1) - \Gamma(\theta_2)|_\infty \le L \|\theta_1 - \theta_2\|_2 \ \ \forall \ \ \theta_1, \theta_2 \in \Theta \,.
$$
\end{assumption}

Let $\Gamma(\Theta)\eqdef \{\Gamma(\theta),\;\theta\in\Theta\}$. Under the data generating process assumed in H\ref{H1}, we consider the constrained empirical risk minimization estimator \begin{equation}
\label{eq:est}
\hat u = \argmin_{u \in \Gamma(\Theta)} \frac1n \sum_{i = 1}^n\left(Y_i - u({\bf s}_i)\right)^2 \,.
\end{equation}
The following theorem is proved in Section \ref{sec:proof:thm:main_param}.

\begin{theorem}
\label{thm:main_param}
Assume H\ref{H1}, and H\ref{assm:gamma_smooth}. Then any solution $\hat u$ of \eqref{eq:est} satisfies
$$
|\hat u - u_\star|_2^2 = O_p\left(\frac{d}{n}\log{\left(\frac{n}{d}\right)}\right) \,.
$$
\end{theorem}

\medskip

It is immediate from the above theorem that $|\hat u - u_\star|_2 = O_p(\sqrt{\frac{d}{n}})$ (up to a log factor) which is the standard parametric rate for estimating a $d$ dimensional parameter. The estimator $\hat u$ is typically not computable since it requires solving the pde for each $\theta\in\Theta$. But the rate in Theorem \ref{thm:main_param}, in light of Theorem \ref{thm:prio:conc}, suggests that the PINN posterior distribution may actually have a convergence rate faster than the nonparametric minimax rate. More research is needed on this issue.

Regarding Assumption \ref{assm:gamma_smooth}, one may replace the $L_2$ norm in the assumption by any $L_p$ norm without hurting the rate, but to translate the complexity of $\Theta$ to $\Gamma(\Theta)$, we believe that some sort of smoothness assumption is necessary. Whether one can do entirely without Assumption \ref{assm:gamma_smooth} is an intriguing question that we leave for future research. 

\subsection{Checking Assumption H\ref{H:lip}-(2)}
We end this section with a result showing that sufficiently smooth functions typically satisfy H\ref{H:lip}-(2). The proof is given in Section \ref{sec:proof:thm:approx:deriv}. We recall that $\Omega$ is a bounded open subset of $\rset^m$, and we write $\mathbf{B}(x,\alpha)$ to denote the Euclidean ball of $\rset^m$ centered at $x$ and with radius $\alpha$. We define $\textsf{int}_\alpha(\Omega)\eqdef\{x\in\Omega:\; \mathbf{B}(x,\alpha)\subset\Omega\}$. We recall also that $\Cset^\beta(\Omega,b)$ denotes the ball with radius $b$ in the Holder space $\Cset^\beta(\Omega)$.

\begin{theorem}\label{thm:approx:deriv}
Given $\varepsilon>0$, $\beta>1$, integer $\tau\in[0,\beta)$,  and $M>0$, define 
\[\bar\alpha =\left(\frac{\tau\varepsilon \lfloor \beta \rfloor!}{2M(\beta-\tau)}\right)^{1/\beta}.\] 
For all $u,\tilde u \in \Cset^{\beta}(\Omega ,M)$ such that $\sup_{x\in\Omega}\;|u(x) - \tilde u(x)|_\infty \leq \varepsilon$, and for all multi-index ${\bf k}=(k_1,\ldots,k_m)$, with $|{\bf k}| \leq \tau$, we have
\[\sup_{x\in\mathsf{int}_{\bar\alpha}(\Omega)}\;\left|D^{\bf k} u(x) - D^{\bf k} \tilde u(x) \right|_\infty \leq C M^{\frac{\tau}{\beta}} \varepsilon^{\frac{\beta-\tau}{\beta}},\]
for some constant $C$ that depends only on $\tau$ and $\beta$.
\end{theorem}

\begin{remark}
We note that $\bar\alpha\downarrow 0$ as $M\uparrow\infty$. Hence if $u,\tilde u \in \Cset^{\beta}(\bar\Omega ,M_0)$ for some well-chosen $\bar\Omega\supset\Omega$, and $\sup_{x\in\bar\Omega}\;|u(x) - \tilde u(x)|_\infty \leq \varepsilon$  then by taking $M\geq M_0$ large enough such that  $\Omega\subset \mathsf{int}_{\bar\alpha}(\bar\Omega)$, we get 
\[\sup_{x\in\Omega}\;\left|D^{\bf k} u(x) - D^{\bf k} \tilde u(x) \right|_\infty \leq C M^{\frac{\tau}{\beta}} \varepsilon^{\frac{\beta-\tau}{\beta}}.\]
\end{remark}

\section{A Bernstein-von-Mises theorem for the $\theta$-marginal}\label{sec:bvm}
Given $\mu\in\rset^d$, and $\Lambda\in\rset^{d\times d}$ symmetric and positive definite, we write $\textbf{N}(\mu,\Lambda)(\cdot)$ as the probability measure of the Gaussian distribution $\textbf{N}(\mu,\Lambda)$. Given $W\in\rset^q$, let 
\begin{equation}\label{eq:conditionW:dist}
\hat\theta_W\eqdef \Sigma_W^{-1}\Phi_W.
\end{equation}
It is then easy to see from (\ref{post:joint:ideal}) that the conditional distribution of $\theta$ given $W$ is $\mathbf{N}(\hat\theta_W,\Sigma_W^{-1}/\lambda)$. For clarity's sake we will write $\Pi^{(\theta)}$ (resp. $\Pi^{(W)}$) to denote the marginal distribution of $\theta$ (resp. $W$) under (\ref{post:joint:ideal}). We have
\[\Pi^{(\theta)}(\cdot\vert\D) =\int_{\rset^q}\left[\mathbf{N}\left(\hat\theta_W,\frac{1}{\lambda}\Sigma_W^{-1}\right)(\cdot)\right]\Pi^{(W)}(\rmd W\vert \D).\]
We set
\begin{equation}\label{eq:target:dist}
\Pi^{(\theta)}_\star(\cdot) \eqdef \mathbf{N}\left(\theta_\star,\frac{1}{\lambda}\Sigma_\star^{-1}\right)(\cdot).
\end{equation}

We investigate the proximity between $\Pi^{(\theta)}$ and  $\Pi_\star^{(\theta)}$ for $n$ large,  using the 2-Wasserstein metric that we denote $\mathsf{W}_2$.  The following result is established in Section \ref{sec:proof:thm:marg:theta}.

\begin{theorem}\label{thm:marg:theta}
Assume H\ref{H1}-H\ref{H:entropy} and the notations of Theorem \ref{thm:margW}. Then, for all $n$ large enough, we have
\[\PE\left[\mathsf{W}_2^2(\Pi^{(\theta)}, \Pi^{(\theta)}_\star)\right] \leq c\left(\r^{2\kappa} + \lambda e^{-n\epsilon_0^2/(2\sigma^2)}  + \lambda e^{-V_1\log(V_2\sqrt{n})/C}\right),\]
for some constant $c$. In particular,
\[\PE\left[\left|\int_{\rset^d} u \Pi^{(\theta)}(\rmd u\vert\D) -\theta_\star \right|^2\right] \leq c\left(\r^{2\kappa} + \lambda e^{-n\epsilon_0^2/(2\sigma^2)}  + \lambda e^{-V_1\log(V_2\sqrt{n})/C}\right).\]
\end{theorem}

\medskip

For DL functions with sufficiently smooth piecewise polynomial activation functions, and for $u_\star\in \Cset^\beta(\Omega)$,  we have seen above that $\r\lesssim n^{-\beta/(m+2\beta)}$, and by Theorem \ref{thm:approx:deriv}, $\kappa =(\beta-\tau)/\beta$. With these we conclude from Theorem \ref{thm:marg:theta} that the estimation rate of $\theta_\star$ afforded by B-PINN is at least $n^{-2(\beta-\tau)/(m+2\beta)}$.  This estimation rate corresponds to the optimal plug-in rate. Recall that our pde equation is linear in $\theta$. Therefore, the bottleneck in estimating $\theta$ is the difficulty of estimating $u$ and its derivatives up to order $\tau$. As shown above, for a function $u\in\Cset^\beta(\Omega)$, PINN estimates these derivatives at the nonparametric minimax rate $n^{-2(\beta-\tau)/(m+2\beta)}$. Hence, the obtained estimation rate for $\theta_\star$. As conjectured above, PINN likely estimates $u_\star$ at a rate faster than the nonparametric minimax rate. If true, this would translate to a faster estimation rate for $\theta_\star$. More research is needed on this point.

\section{Numerical illustration}\label{sec:illust}
We illustrate the results above using the one-dimensional heat equation of Example \ref{ex:1}. Extensive illustrations of PINN and Bayesian PINN can be found in the literature (\cite{raissi:etal:19,yang:bpinn:21,cuomo:etal:22}). We focus here on comparing PINN and an approach that does not directly exploit the PDE structure of the problem. We consider the heat equation with $L=\pi$, and $T=1$.  Hence $\Omega= (0,1)\times (0,\pi)$. We use the boundary condition $u(t, 0) = u(t, \pi) = 0$, and the initial condition $u(0, x) = \sin ( x)$. Given $\theta$ the heat equation has a unique solution $u_\theta(t,x) =  \sin (x) \exp(-\theta t)$. We set the true value of $\theta$ to $\theta_\star=0.5$.  
The observed data $\D$ is generated as follows: $n$ sensor locations $\{{\bf s}_i=(t_i,x_i),\;1\leq i\leq n\}$ are evenly distributed in $\Omega$, yielding measurements 
\[Y_i = u_{\theta_\star}({\bf s}_i) + \sigma \epsilon_i, \mbox{ where }\; \epsilon_i \sim \mathbf{N}(0,1),\]
for some noise parameter $\sigma$ that we control. For the function class $\F=\{u_W\}$, $u_W:\;\rset^2\to\rset$ is taken to be a fully connected neural network with depth $4$, width $64$, and $\texttt{tanh}$ activation function. The resulting posterior distribution is
\[\Pi(\rmd\theta,\rmd W\vert\D)\propto \exp\left(-\frac{1}{2\sigma^2}\sum_{i=1}^n\left(Y_i - u_W({\bf s}_i)\right)^2 -\frac{\lambda}{2}\ell_0(\theta,W) -\frac{1}{2}\|\theta\|_2^2\right)\Pi_0(\rmd W)\rmd\theta,\]
where
\begin{equation}\label{loss:2}
\ell_0(\theta,W) =  |(u_W)_t -  \theta (u_W)_{xx}|_2^2 +|u_W(\cdot,0)|_2^2+|u_W(\cdot,\pi)|_2^2+|u_W(0,\cdot) - \sin(\cdot)|_2^2,
\end{equation}
where $|\cdot|_2$ denotes the function space $L^2$-norm. Observe that $\lambda=0$ in the above posterior corresponds to the estimation of $W$ without the physics-informed prior (non-PINN estimation). Before presenting the numerical results, we describe briefly our MCMC sampling method.

\subsection{Approximation and MCMC sampling}
In general, the $L^2$-norm in the loss (\ref{loss:2}) is intractable and is typically replaced by a Monte Carlo or a numerical quadrature approximation. Here we use Monte Carlo by drawing $N = 10,000$ interior points $({\bf s}_1^{(i)},\ldots,{\bf s}^{(i)}_N)$ uniformly in $\Omega$ to estimate the $L^2$ norm on $\Omega$. For the $L^2$ norm along the boundaries we draw two sets of $B=128$ points $({\bf s}_1^{(b_i)},\ldots,{\bf s}^{(b_i)}_B)$ for $i=1,2$ uniformly on the time and space boundaries respectively. Hence, we consider  the approximate posterior distribution
\begin{equation}\label{post:joint:hat}
\widehat{\Pi}(\theta,W\vert\D) \propto \Pi_0(W)\\\exp\left(-\frac{1}{2\sigma^2}\sum_{i=1}^n\left(y_i - u_W({\bf s}_i)\right)^2 - \frac{\lambda}{2}\widehat{\ell}_0(\theta,W)- \frac{\rho}{2}\|\theta\|_2^2\right).
\end{equation}
where the loss (\ref{loss:2}) is replaced by
\begin{multline}\label{loss:3}
\widehat{\ell}_0(\theta,W) =  \frac{1}{N}\sum_{k=1}^N \left((u_W)_t({\bf s}_k^{(i)}) -  \theta\times (u_W)_{xx}({\bf s}_k^{(i)})\right)^2 \\
+\frac{1}{B}\sum_{k=1}^{B}\left(u_W({\bf s}^{(b_1)}_k,0)\right)^2 + \frac{1}{B}\sum_{k=1}^{B}\left(u_W({\bf s}^{(b_1)}_k,\pi)\right)^2 +\frac{1}{B}\sum_{k=1}^B\left(u_W(0,{\bf s}_k^{(b_2)}) - \sin({\bf s}_k^{(b_2)})\right)^2.
\end{multline}

Given $W$, and as seen in Section \ref{sec:bvm}, the posterior conditional distribution of $\theta$ given $W$ is 
\begin{equation}\label{cond:theta}
\theta \vert W, \mathcal{D} \sim \mathbf{N}\left(\Sigma_W^{-1}\Phi_W  ,\frac{1}{\lambda}\Sigma_W^{-1}\right)
\end{equation}
where in the particular case of this example 
$$
\Sigma_W = \frac{1}{N} \sum_{k=1}^N (u_W)_{xx}(\mathbf{s}_k^{(i)})^2 + \frac{\rho}{\lambda}, \;\; \mbox{ and } \;\;  \Phi_W = \frac{1}{N}\sum_{k=1}^N (u_W)_t(\mathbf{s}_k^{(i)}) \times (u_W)_{xx}(\mathbf{s}_k^{(i)})
$$

We sample from (\ref{post:joint:hat}) using the approximate asynchronous sampler of \cite{atchade:wang:23}. The algorithm is a data-augmentation Metropolis-with-Gibbs sampler where the update of the sparsity support $\Lambda$ given $\theta,W$ uses asynchronous sampling, and the update of $W$ given $\theta,\Lambda$ is a sparse stochastic gradient Langevin dynamics. Then  $\theta$ given $W,\Lambda$ is sample from Gaussian $ \mathbf{N}\left(\Sigma_{W_\Lambda}^{-1}\Phi_{W_\Lambda}  ,\frac{1}{\lambda}\Sigma_{W_\Lambda}^{-1}\right)$ using (\ref{cond:theta}) where $W_\Lambda$ is the component-wise product of $W$ and $\Lambda$ (a sparse neural network weight).  A \texttt{Pytorch} implementation is available from the GitHub page \texttt{https://github.com/xliu-522/SA-cSGLD}.

Throughout the experiment, we use $\theta^* = 0.5$, $\lambda = n$ and $\rho = 1$. For each instance of the posterior distribution (\ref{post:joint:hat}) under consideration, we run the aforementioned MCMC sampler until convergence and keep running another $200,000$ iterations. We then record every 20th sample, resulting in a total of $K = 10,000$ samples, denoted by  $\{\hat\theta^{(k)} \}_{k=1}^K$ . The mean and standard deviations of the marginal posterior  distribution of $\theta$ are then approximated respectively by  
\[\mu_{\hat\theta} \eqdef \frac{1}{K} \sum_{k=1}^{K} \hat{\theta}^{(k)},\;\mbox{ and }\;\; \sigma_{\hat\theta} \eqdef \sqrt{\frac{1}{K} \sum_{k=1}^{K} \left(\hat{\theta}^{(k)} - \mu_{\hat\theta} \right)^2}.\]

In this section, assuming that the MCMC sampler has converged, we take the distribution of the samples $\{\hat\theta^{(k)} \}_{k=1}^K$ to be $\Pi^{(\theta)}$ and we denote its normal approximation by $\tilde{\Pi}^{(\theta)} \eqdef \mathbf{N}(\mu_{\hat\theta}, \sigma^2_{\hat\theta})$.
Since the exact solution $u_\star(t,x) = \sin (x) \exp(-0.5 t)$  is known, we can compute the variance of the limiting Gaussian distribution in Bernstein-von Mises theorem (Theorem \ref{thm:marg:theta})
$$
\Sigma_* = \int_0^1 \int_0^\pi \left((u_\star)_{xx}(t, x) \right)^2 dx dt + \frac{1}{n} = \frac{\pi (1 - \frac{1}{e})}{2} + \frac{1}{n} \approx 0.993 + \frac{1}{n}
$$
Hence the limiting distribution in Theorem \ref{thm:marg:theta} denoted $\Pi_\star^{(\theta)}$ is 
\begin{equation}\label{heat:target:normal}
	\Pi_\star^{(\theta)} \eqdef \mathbf{N}(\theta_\star, \frac{1}{n}  \Sigma_\star^{-1}).
\end{equation}

\subsection{PINN versus non-PINN comparison}
In this section, we compare  PINN to a two-step approach (denoted by non-PINN) that does not directly use the PDE structure of $u$. Specifically, in the non-PINN approach, we fit the same DL model for estimating $u$ but without the PINN prior. The resulting posterior is given by 
$$
\widehat{\Pi}( W \vert\D) \propto \Pi_0(W)\exp\left(-\frac{1}{2\sigma^2}\sum_{i=1}^n\left(y_i - u_W({\bf s}_i)\right)^2  \right).
$$
We sample from this posterior distribution using the same MCMC algorithm described above. For each MCMC draw $W$, we then subsequently solve for $\theta$ by sampling from the Gaussian distribution $\mathbf{N}\left(\Sigma_W^{-1}\Phi_W, \frac{1}{n}\Sigma_W^{-1}\right)$ from (\ref{cond:theta}). Hence, in the two-step approach, the PDE information is not used in the first step for estimating $u$, but is used in the second step to recover $\theta$ by the linear regression model derived from (\ref{main:pde}). By contrast, PINN uses the PDE information and estimates $(u,\theta)$ jointly. We compare the two approaches under varying sample sizes and noise levels, as it allows us to observe the contribution of the PINN prior $\ell_0(\theta,W)$. 

Figure \ref{sec:heat:fig:1} provides a comparative analysis of the posterior distribution of $\theta$ in the two approaches (PINN and non-PINN) at a noise level of 10\% ($\sigma \approx 0.025$), across different sample sizes. Both methods exhibit improved parameter estimates as the sample size $n$ increases, as indicated by the decreasing bias and interquartile range (IQR). However, training with the PDE term generally yields more concentrated distributions of $\theta$ with narrower IQRs, indicating greater stability and reliability in parameter estimation. This demonstrates that incorporating the PDE term results in more accurate and robust parameter estimates.

\medskip

\begin{figure}[h]
	\centering
	\includegraphics[width=1.0\textwidth]{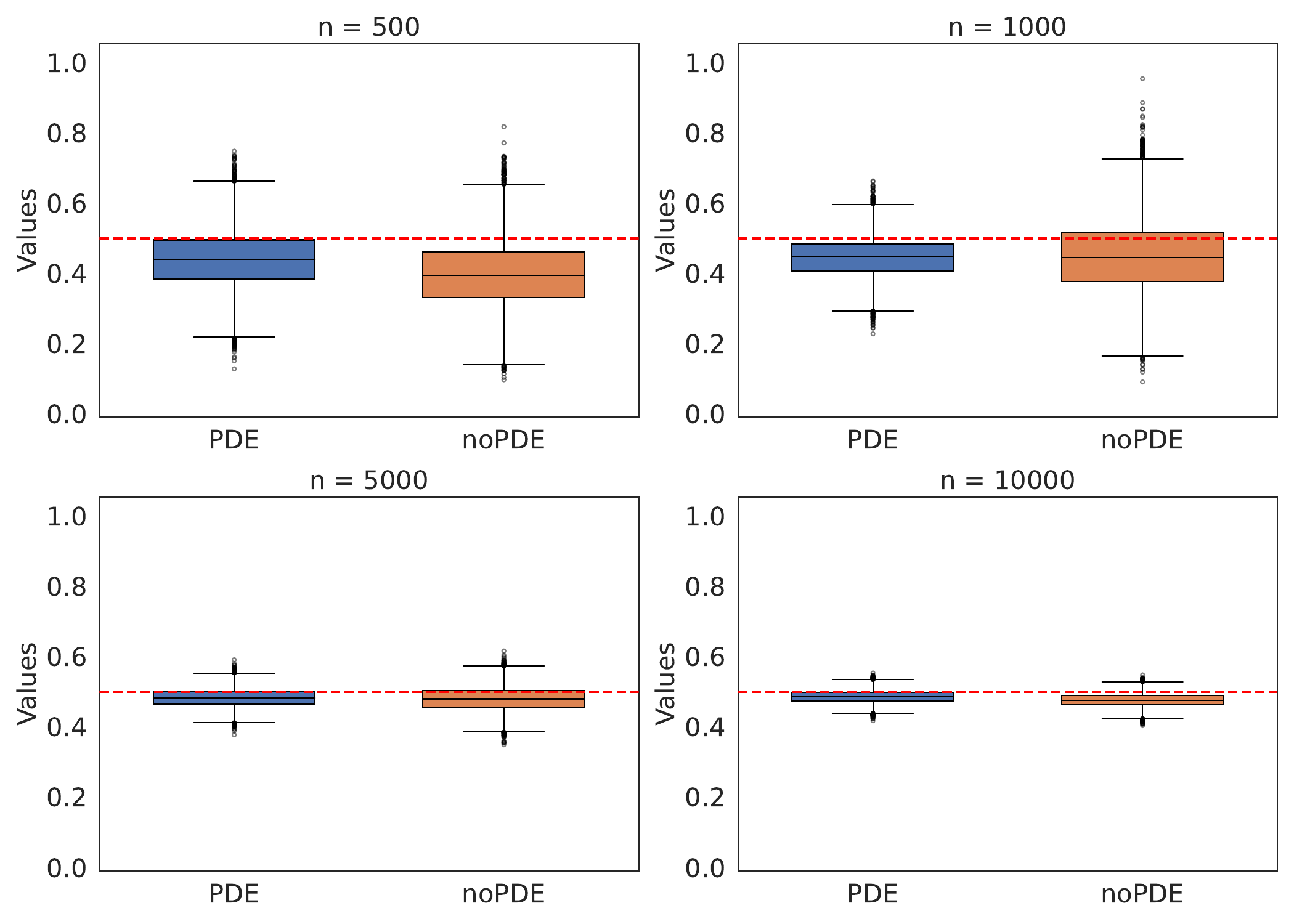} 
	\caption{Comparison between boxplots of sampled $\theta$ from training with/without pde with different sample sizes n at noise level = 10\%.  True $\theta^* = 0.5$ in the dotted line.}
	\label{sec:heat:fig:1}
\end{figure}

To further investigate the role of the PDE term in the recovery of $\theta$, the root mean square error(rMSE) and the Wasserstein-2 distance ($W_2$) are computed and summarized in Table \ref{sec:heat:table:1} and Figure \ref{sec:heat:fig:2}. The rMSE is defined as $\sqrt{\frac{1}{K} \sum_{k=1}^K(\hat{\theta}^{(k)} - \theta^*)^2}$, whereas the $W_2$ metric is  computed using the \texttt{ot.emd2\_1d} function from the POT library by \cite{flamary2021pot} with the parameter $p=2$. These metrics provide a comprehensive comparison of the accuracy (rMSE) and distribution similarity ($W_2$) between sample $\theta$ and the true value $\theta^*$.

\begin{figure}[h]
	\centering
	\includegraphics[width=1.0\textwidth]{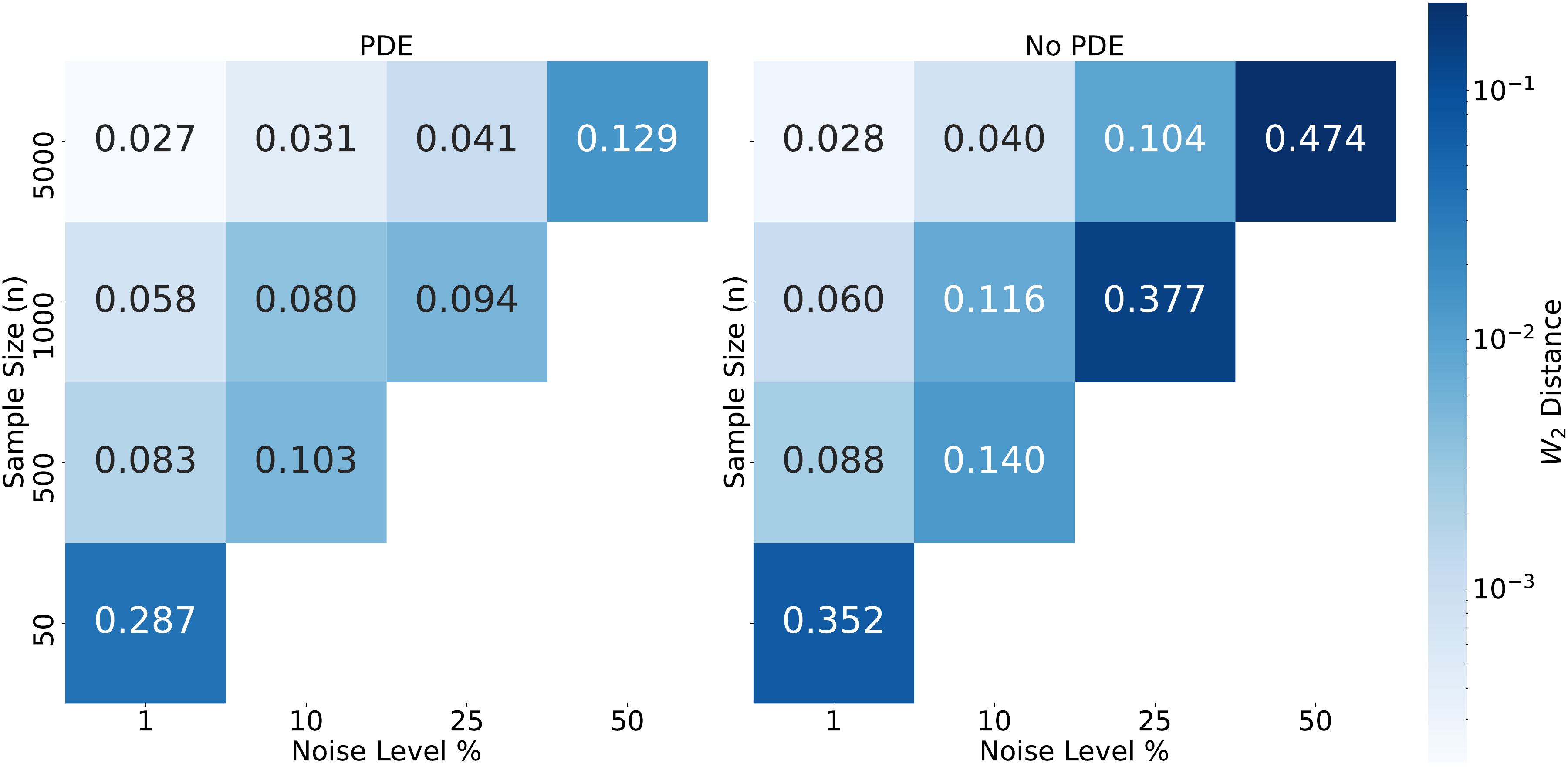} 
	\caption{Comparison between training with/without pde term. The number in each grid is the rMSE between sampled $\theta$ and $\theta^*$. The color represents $W_2$ distance $W_2(\Pi^{(\theta)}, \Pi_*^{(\theta)})$  }
	\label{sec:heat:fig:2}
\end{figure}

\begin{table}[h]
	\centering
	\small
	\caption{Comparison of PDE and No-PDE under different noise levels}
	\label{sec:heat:table:1}
	\begin{tabular}{|c|c||c|c||c|c|}
		\hline
		&  & \multicolumn{2}{c||}{rMSE} & \multicolumn{2}{c|}{$W_2(\Pi^{(\theta)}, \Pi_*^{(\theta)})$} \\
		\hline
		Noise Level & n  & \textbf{PDE} &\textbf{No-PDE} &  \textbf{PDE} &\textbf{No-PDE}  \\
		\hline
		1\% & 50    & 0.287   & 0.352 & 0.0379 & 0.0699 \\
		1\% & 500   & 0.083   & 0.088 & 0.0018& 0.0024 \\
		1\% & 1000  & 0.058   & 0.060& 0.0008 & 0.0010 \\
		1\% & 5000 & 0.027   & 0.028 & 0.0002& 0.0003 \\
		10\% & 500  & 0.103   & 0.140 & 0.0052& 0.0131 \\
		10\% & 1000  & 0.080   & 0.116 & 0.0036& 0.0078 \\
		10\% & 5000 & 0.031  & 0.040 & 0.0004 & 0.0008 \\
		10\% & 10000  & 0.023  & 0.031 & 0.0002 & 0.0007 \\
		25\% & 1000  & 0.094   & 0.377 & 0.0055& 0.1387 \\
		25\% & 5000  & 0.042  & 0.104 & 0.0011 & 0.0090 \\
		50\% & 5000  & 0.129  & 0.474 & 0.0147& 0.2246 \\
		\hline
	\end{tabular}
\end{table}

Figure \ref{sec:heat:fig:2} and Table \ref{sec:heat:table:1} illustrate the performance comparison of models trained with and without PDE constraints across various noise levels and sample sizes.  
As the number of samples ($n$) increases or the noise level decreases, the values of rMSE and $W_2$ distance decrease for both training methods, with and without the PDE constraint. 
 However, the PDE-constrained models perform better, demonstrating lower rMSE and $W_2$ distances than corresponding models without PDE constraints. The difference in performance becomes more noticeable at low sample sizes or high noise levels. This suggests that the PDE constraint can help the model learn more effectively when there is insufficient information from the data alone. 

\subsection{Posterior contraction Behavior}
From the previous section and Table \ref{sec:heat:table:1}, we observed that the $W_2$ distance between $\Pi^{(\theta)}$ and the limiting distribution $\Pi_\star^{(\theta)}$ given in (\ref{heat:target:normal}) decays significantly as the sample size $n$ increases, which is consistent with the conclusion of Theorem \ref{thm:marg:theta}. To further study this contraction behavior, we plot the histogram of $\Pi^{(\theta)}$ (using samples from the MCMC sampler), its Gaussian approximation $\tilde\Pi^{(\theta)}=\textbf{N}(\mu_{\hat\theta}, \sigma^2_{\hat\theta})$ where $\mu_{\hat\theta}$ and $ \sigma_{\hat\theta}$ are computed from the MCMC samples, and the limiting distribution $\Pi_\star^{(\theta)}$,  for various sample sizes ($n = 500$, $1000$, $5000$, $10000$) at a noise level of 10\% ($\sigma \approx 0.025$). The perfect match between $\Pi^{(\theta)}$ and $\tilde\Pi^{(\theta)}$ in Figure \ref{sec:heat:fig:3} suggests that the posterior distribution of $\theta$ is approximately Gaussian. However the persistent discrepancy between $\Pi^{(\theta)}$ and $\Pi_\star^{(\theta)}$ suggests that the $W_2$ convergence to zero between $\Pi^{(\theta)}$ and $\Pi_\star^{(\theta)}$ (as established in Theorem \ref{thm:marg:theta} and illustrated above), likely do not hold in total variation. For two probability measures $P,Q$ with densities $p,q$ respectively, their total variation distance is $\text{TVD}(P, Q) \eqdef \frac{1}{2} \int |p(x) - q(x)| \, dx$. To further illustrate this point we compute in Table \ref{sec:heat:table:2} a lower and an upper on the total variation distance between $\tilde \Pi^{(\theta)}$ and $\Pi_\star^{(\theta)}$.  We compute these bounds by noting (see \cite{devroye2023total}) that for any two probability measures $P,Q$ with densities $p,q$ respectively we have
\begin{equation}\label{lu:tv}
H^2(P, Q)  \leq \text{TVD}(P, Q) \leq \min\left(1, \sqrt{KL(P \lVert Q) / 2}, \sqrt{KL(Q \lVert P) / 2}\right),
\end{equation}
where in the univariate Gaussian case where $P \sim \mathbf{N}(\mu_1, \sigma_1^2)$ with density $p(x)$  and  $Q \sim \mathbf{N}(\mu_2, \sigma_2^2)$ with density $q(x)$, the Hellinger distance $H(P, Q)$ is (see e.g. \cite{pardo2006statistical})
\[H^2(P, Q) \eqdef \frac{1}{2}\int \left(\sqrt{p(x)} - \sqrt{q(x)}\right)^2 \, dx = 1 - \sqrt{\frac{2\sigma_1\sigma_2}{\sigma_1^2 + \sigma_2^2}} \exp\left( -\frac{(\mu_1 - \mu_2)^2}{4(\sigma_1^2 + \sigma_2^2)} \right),\]
and the $KL$-divergence is (see e.g. \cite{rasmussen2005gaussian}) 
\[
D_{KL}(P \| Q) \eqdef \int p(x) \log \frac{p(x)}{q(x)} \, dx = \log\left(\frac{\sigma_2}{\sigma_1}\right) + \frac{\sigma_1^2 + (\mu_1 - \mu_2)^2}{2\sigma_2^2} - \frac{1}{2}.\]

Using these formulas, the third and fourth column of Table \ref{sec:heat:table:2} shows the left-hand side and right-hand side of (\ref{lu:tv}) respectively. From this table, we see, for instance, that at $10\%$ noise level, the $W_2$ distance decreases with the sample size, whereas the TVD lower and upper bounds do not. This lack of TV convergence is because the posterior mean does not converge to $\theta_\star$ fast enough. Indeed, given two univariate normal distributions $\textbf{N}(\mu_{1,n},\frac{1}{n})$ and $\textbf{N}(\mu_{2,n},\frac{1}{n})$, as $n\to\infty$, their $W_2$ distance converges to $0$ as soon as $|\mu_{1,n} - \mu_{2,n}|\to 0$, whereas their Hellinger distance (using the formula above) converges to zero if and only if $|\mu_{1,n} - \mu_{2,n}|$ converges to zero faster than $1/\sqrt{n}$.


\begin{figure}[h]
	\centering
	\includegraphics[width=1.0\textwidth]{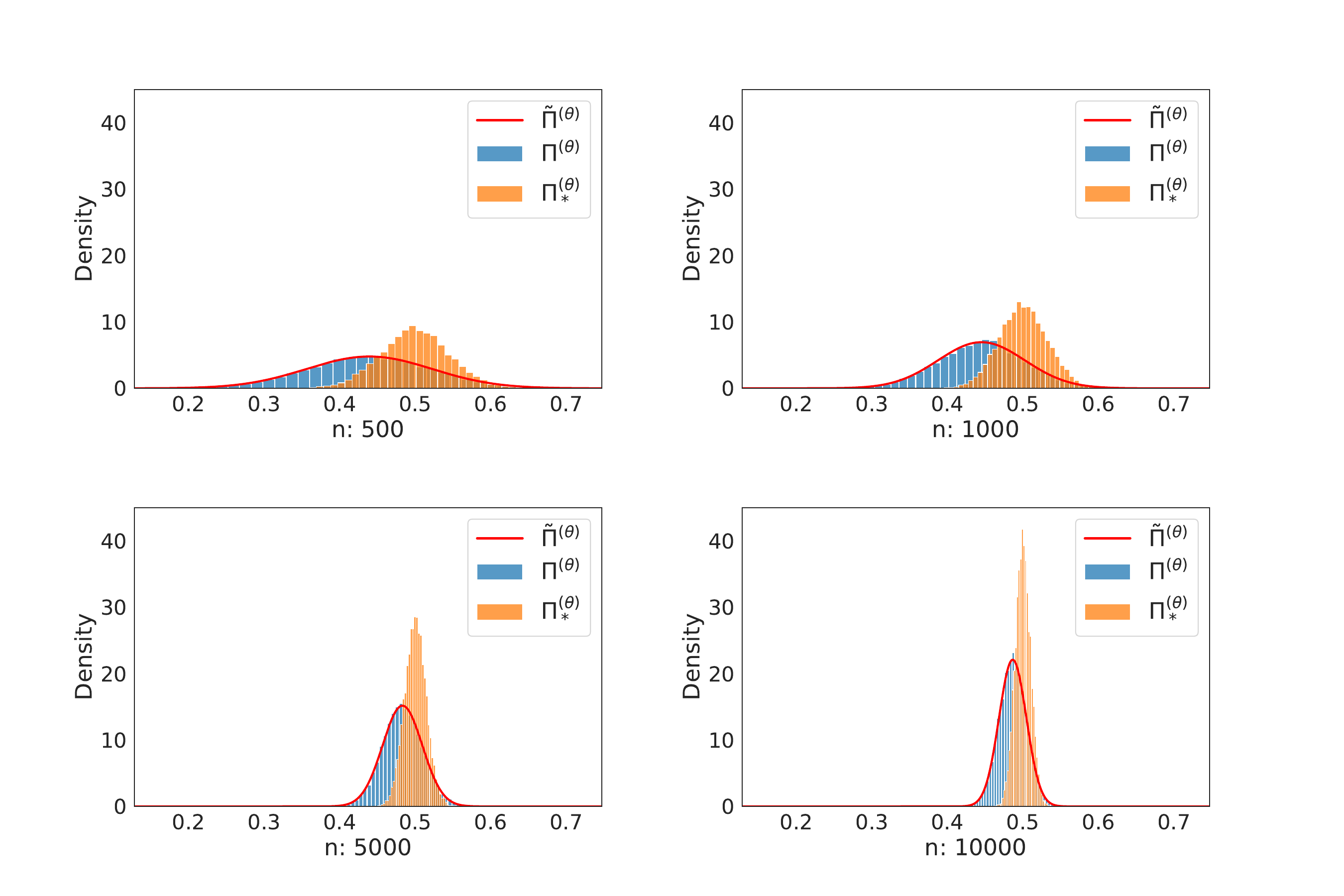} 
	\caption{Histograms of sampled $\theta$, denoted $\Pi^{(\theta)}$, in blue with its normal approximation $\tilde{\Pi}^{(\theta)}$ in red line and target Gaussian distribution $\Pi_*^{(\theta)}$ in orange with increasing sample size n at noise level 10\%}
	\label{sec:heat:fig:3}
\end{figure}



\begin{table}[h!]
	\centering
	\small
	\caption{ $W_2$ and upper and lower bound of TVD under different sample sizes and noise levels}
	\label{sec:heat:table:2}
	\begin{tabular}{|c|c||c|c|c|}
		\hline
		{Noise Level} & {n} & {$W_2$} & {TVD Lower} & {TVD Upper} \\
		\hline
		1\%  & 50    & 0.0377 &  0.149 & 0.482  \\
		1\%  & 500   & 0.0018 & 0.090 & 0.372  \\
		1\%  & 1000  & 0.0008 &  0.084& 0.360 \\
		1\%  & 5000  & 0.0002 & 0.099& 0.387  \\
		10\% & 500   & 0.0051 & 0.176&0.515  \\
		10\% & 1000  & 0.0037 & 0.228&0.594  \\
		10\% & 5000  & 0.0004 & 0.156&0.485   \\
		10\% & 10000 & 0.0002 & 0.173 &0.512 \\
		25\% & 1000  & 0.0054 & 0.270&0.637   \\
		25\% & 5000  & 0.0011 &  0.287 &0.665 \\
		50\% & 5000  & 0.0148 & 0.616&1.000   \\
		\hline
	\end{tabular}
\end{table}

\appendix
\section{Proof of Theorem \ref{thm:margW}}\label{sec:proof:thm:margW}
We define 
\[\Delta_W(y,{\bf s}) \eqdef \left(y-u_W({\bf s})\right)^2,\;\; \Delta_\star(y,{\bf s}) \eqdef \left(y-u_\star({\bf s})\right)^2,\;\;(y,{\bf s})\in\rset\times\Omega.\]
We recall that $\PP_n(f) = n^{-1}\sum_{i=1}^n f(Y_i,{\bf s}_i)$. It follows from (\ref{marginalW}) that for any set $C\subseteq\rset^q$, we have
\begin{equation}\label{marginalW:2}
\Pi(C\vert \D) = \frac{\int_C \exp\left(-\frac{n}{2\sigma^2}\PP_n(\Delta_W-\Delta_\star) + \mathcal{R}_W\right)\Pi_0(\rmd W)}{\int_{\rset^q} \exp\left(-\frac{n}{2\sigma^2}\PP_n(\Delta_W-\Delta_\star) + \mathcal{R}_W\right)\Pi_0(\rmd W)},\end{equation}
where
\begin{multline*}
\mathcal{R}_W \eqdef \mathcal{R}_W^{(1)} + \mathcal{R}_W^{(2)},\;\;\mbox{ with }\;\; \mathcal{R}_W^{(1)} \eqdef -\frac{1}{2}\left(\log\det(\Sigma_W) - \log\det(\Sigma_\star)\right),\;\;\\
\;\;\mbox{ and }\;\; \mathcal{R}_W^{(2)} \eqdef -\frac{\lambda}{2}{\cal J}(u_W).\end{multline*}
We break the proof into three parts. First, in Section \ref{sec:lb:nc}, we give a lower bound on the normalizing constant of the posterior distribution as given in (\ref{marginalW:2}). Then we show in Section \ref{sec:prior:concentration} that the prior distribution $\Pi_0$ has a good inductive bias, as it puts a high probability on $W$ that is sparse. The third part of the proof in Section \ref{sec:emp:proc} establishes some deviation bounds for the empirical process of the log-likelihood ratio. Then, we put all the pieces together in Section \ref{sec:finish:proof:thm:margW}.

\subsection{Lower bound on the normalizing constant}\label{sec:lb:nc}

\begin{lemma}\label{lem:lb}
Assume H\ref{H1}-H\ref{H:stability}. Let $\epsilon_0,s_0,W_0$ be as in H\ref{H:approx}. Then for all $n$ large enough, with probability at least $1 - 2e^{-n\epsilon_0^2/(2\sigma^2)}$, we have
\[
\int_{\rset^q} \exp\left(-\frac{n}{2\sigma^2}\PP_n(\Delta_W-\Delta_\star) + \mathcal{R}_W\right)\Pi_0(\rmd W) \geq C\times  \exp\left(-\frac{Cn\epsilon_0^2}{\sigma^2}\right),
\]
for some absolute constant $C$.
\end{lemma}
\begin{proof}
By definition of $W_0$, $|u_{W_0} - u_\star|_\infty \leq \epsilon_0$. Let $L\eqdef L_{W_0,1}\geq 1$ be as in H\ref{H:lip}. Let $\Lambda_0\in\{0,1\}^q$ denote the sparsity support of $W_0$,  and let 
\begin{equation}\label{eq:eta}
\eta \eqdef \frac{\epsilon_0}{L s_0^{1/2}},\;\;\mbox{ and }\;\; \mathcal{V}\eqdef\{W\in\rset^q:\; \mathsf{supp}(W)=\Lambda_0,\;\|W - W_0\|_\infty \leq \eta\}.
\end{equation}
Since $\eta s_0^{1/2}\leq 1$, for all $W\in\mathcal{V}$, we have
\begin{equation}\label{l2:ball}
\|W - W_0\|_2 \leq s_0^{1/2}\|W - W_0\|_\infty \leq s_0^{1/2}\eta \leq 1.
\end{equation}
Using (\ref{l2:ball}), and the Lipschitz assumption imposed in H\ref{H:lip}, we deduce that for all $W\in\mathcal{V}$,
\begin{equation}\label{ballN}
|u_W - u_{W_0}|_\infty\leq L\|W- W_0\|_2 \leq L s_0^{1/2}\eta\leq \epsilon_0.
\end{equation}
From the above display, and appealing to (\ref{lip:H}) and H\ref{H:lip}-(3) we further deduce that for all $W\in\mathcal{V}$,
\begin{equation}\label{ballN:1}
| \H u_W - \H u_{W_0}|_2  \leq  C_0 \max_{{\bf k}:\;|{\bf k}|\leq \tau} \;\left|D^{\bf k} u_W - D^{\bf k} u_{W_0}\right|_\infty \leq C_0c_1 \epsilon_0^\kappa.
\end{equation}
In view of H\ref{H:stability}, we also deduce that for all $W\in\mathcal{V}$,
\begin{equation}\label{ballN:2}
|u_W - u_\star|_\infty\leq 2\epsilon_0, \;\;\mbox{ and }\;\; | \H u_W - \H u_\star|_2  \leq 2^\kappa C_4 \epsilon_0^\kappa.
\end{equation}

\paragraph{\texttt{Step 1}} First we show that with probability at least $1 - 2e^{-n\epsilon_0^2/(2\sigma^2)}$, it holds
\begin{equation}\label{step:1}
\int_{\rset^q} \exp\left(-\frac{n}{2\sigma^2}\PP_n(\Delta_W-\Delta_\star) + \mathcal{R}_W\right)\Pi_0(\rmd W) \geq  \frac{1}{2}e^{-4n\epsilon_0^2/\sigma^2}\int_{\mathcal{V}} e^{\mathcal{R}_W} \Pi_0(\rmd W).
\end{equation}
To establish this, we recall that with $\xi_i = (Y_i - u_\star({\bf s}_i))/\sigma$, we have
\[-\frac{n}{2\sigma^2}\PP_n(\Delta_W-\Delta_\star) = -\frac{1}{2\sigma^2}\sum_{i=1}^n (u_W({\bf s}_i) - u_\star({\bf s}_i))^2 +\frac{1}{\sigma}\sum_{i=1}^n \xi_i(u_W({\bf s}_i) - u_\star({\bf s}_i)).\]
Using (\ref{ballN:2}), we note that for $W\in\mathcal{V}$, 
\[\frac{1}{\sigma^2}\sum_{i=1}^n (u_W({\bf s}_i) - u_\star({\bf s}_i))^2 \leq A,\]
where
\[A \eqdef \frac{4n\epsilon_0^2}{\sigma^2}.\]
By Gaussian tail bounds, for all $W\in\mathcal{V}$,
\[\PP\left(\frac{1}{\sigma}\sum_{i=1}^n \xi_i(u_W({\bf s}_i) - u_\star({\bf s}_i)) > \frac{A}{2} \;\vert {\bf s}_{1:n}\right)\leq \exp\left(-\frac{A^2}{8A}\right)\leq e^{-A/8}.\]
Therefore, with $\e_W \eqdef \{\D:\; \frac{1}{\sigma}\sum_{i=1}^n \xi_i(u_W({\bf s}_i) - u_\star({\bf s}_i)) \leq A/2\}$,
\begin{multline*}
\int_{\rset^q} \exp\left(-\frac{n}{2\sigma^2}\PP_n(\Delta_W-\Delta_\star) + \mathcal{R}_W\right)\Pi_0(\rmd W) \geq e^{-A} \int_{\mathcal{V}}\textbf{1}_{\e_W}(\D) e^{\mathcal{R}_W}\Pi_0(\rmd W),
\end{multline*}
and with $\e_W^c$ denoting the complement of $\e_W$, we obtain
\begin{multline*}
\PP\left[\int_{\rset^q} \exp\left(-\frac{n}{2\sigma^2}\PP_n(\Delta_W-\Delta_\star) + \mathcal{R}_W\right)\Pi_0(\rmd W) < \frac{e^{-A}}{2}\int_{\mathcal{V}} e^{\mathcal{R}_W} \Pi_0(\rmd W) \;\vert {\bf s}_{1:n}\right] \\
\leq \PP\left[\int_{\mathcal{V}}\textbf{1}_{\e_W}(\D)e^{\mathcal{R}_W} \Pi_0(\rmd W) < \frac{1}{2}\int_{\mathcal{V}} e^{\mathcal{R}_W} \Pi_0(\rmd W) \;\vert {\bf x}_{1:n}\right] \\
= \PP\left[\int_{\mathcal{V}}\textbf{1}_{\e_W^c}(\D)e^{\mathcal{R}_W}  \Pi_0(\rmd W) > \frac{1}{2}\int_{\mathcal{V}}e^{\mathcal{R}_W} \Pi_0(\rmd W)  \;\vert\; {\bf x}_{1:n}\right]\\
\leq \frac{2}{\int_{\mathcal{V}}e^{\mathcal{R}_W} \Pi_0(\rmd W)}\int_{\mathcal{V}}\PP\left(\frac{1}{\sigma}\sum_{i=1}^n \xi_i(u_W({\bf s}_i) - u_\star({\bf s}_i)) > \frac{A}{2}\;\vert\; {\bf s}_{1:n}\right)e^{\mathcal{R}_W} \Pi_0(\rmd W)\\
\leq 2e^{-A/8},
\end{multline*}
which is (\ref{step:1}).

\paragraph{ \texttt{Step 2}} We now show that
\begin{equation}\label{eq:lb:d}
\int_{\mathcal{V}}e^{\mathcal{R}_W} \Pi_0(\rmd W) \geq c \times \exp\left(-C \frac{n\epsilon_0^2}{\sigma^2}\right).
\end{equation}
Using the definition of $\Sigma_W$, we first note that for all $W_1,W_2$ 
\begin{eqnarray}\label{comp:Sigma}
\|\Sigma_{W_1} - \Sigma_{W_2}\|_{\textsf{op}} & \leq & \sqrt{2}\left(|\H u_{W_1}|_2 + |\H u_{W_2}|_2\right)\; |\H u_{W_1} - \H u_{W_2}|_2 \nonumber\\
& \leq & \sqrt{2}\left(2|\H u_{W_1}|_2 + |\H u_{W_2} - \H u_{W_1}|_2 \right)|\H u_{W_1} - \H u_{W_2}|_2.\end{eqnarray}
We combine this with (\ref{ballN:1}) to obtain that  for all $W\in\mathcal{V}$,
\[
\|\Sigma_{W} - \Sigma_{W_0}\|_{\textsf{op}} \leq 2C_0c_1\left(2|\H u_{W_0}|_2 + C_0c_1\epsilon_0^\kappa\right)\epsilon_0^\kappa \leq \frac{C_2}{2},\]
for all $n$ large enough, since $\epsilon_0\to 0$, as $n\to\infty$. Therefore, Weyl's inequality, and $\lambda_{\textsf{min}}(\Sigma_{W_0})\geq C_2$ imply that for all $W\in\mathcal{V}$, and $n$ large enough
\[\lambda_{\textsf{min}}(\Sigma_{W})\geq \frac{C_2}{2}.\]
As a result of the last display, we can use a first order Taylor expansion of the $\log\det$ to conclude that for all $W\in\mathcal{V}$,
\begin{multline*}
|\mathcal{R}_W^{(1)}|   = \left|\frac{1}{2}\left(\log\det(\Sigma_W) - \log\det(\Sigma_\star)\right)\right| \leq \frac{2d^{1/2}}{C_2}\normfro{\Sigma_W - \Sigma_\star} \leq \frac{2d}{C_2}\|\Sigma_W - \Sigma_\star\|_{\textsf{op}}\\
\leq \frac{4d}{C_2}\left(2|\H u_\star|_2 + |\H u_W - \H u_\star|_2\right)|\H u_W - \H u_\star|_2\\
 \leq 4\frac{2^\kappa C_4 d}{C_2} \left(2|\H u_\star|_2 +  2^\kappa C_4 \epsilon_0^\kappa\right)\epsilon_0^\kappa \leq 1,
\end{multline*}
again, for all $n$ large enough. For $W\in\mathcal{V}$, let $J_W$ be the $L^2$ projector on the linear space spanned by the function $\{(\H u_W)_i,\;1\leq i\leq d\}$ in $L^2(\Omega,\rset,\nu)$. Note that, since $\lambda_{\textsf{min}}(\Sigma_W)\geq C_2/2$, that sub-space is isomorphic to $\rset^d$, and by expressing the calculation in $\rset^d$, it easily follows that
\[\left|\Phi_W^\t \Sigma_W^{-1}\Phi_W - |J_W \bar f_W|_2^2\right| \leq \frac{2}{\lambda C_2},\]
where we recall that $\bar f_W = f -\H_0 u_W$. 
As a result, for $W\in \mathcal{V}$, 
\[\left|\Phi_W^\t \Sigma_W^{-1}\Phi_W - |\bar f_W|_2^2 \right| \leq \frac{2}{\lambda C_2} + \left||J_W \bar f_W|_2^2 -  |\bar f_W|_2^2\right| = \frac{2}{\lambda C_2} + \left|J_W \bar f_W - \bar f_W\right|_2^2,\]
where the equality uses the fact that for all $u\in L^2(\Omega,\rset,\nu)$,  $|u|_2^2 = |J_W u|_2^2 + |(J_W-\un)u|_2^2$. By the definition of the projector as closest element, and since $\bar f_W = f - \H_0 u_W = (\H_0 u_\star - \H_0 u_W) + \theta_\star^\t\H u_\star$,
\begin{multline*}
\left|J_W \bar f_W - \bar f_W\right|_2^2 \leq \left|\theta_\star^\t(\H u_W) -\bar f_W \right|_2^2 \leq  2\left|\H_0 u_\star - \H_0 u_W\right|_2^2 + 2\left|\theta_\star^\t(\H u_W) - \theta_\star^\t(\H u_\star)\right|_2^2 \\
\leq  C \left(1 + \|\theta_\star\|_2^2\right) \epsilon_0^{2\kappa},
\end{multline*}
for some absolute constant $C$ that depends only on  $\kappa$ and $C_4$. We conclude  that for $W\in\mathcal{V}$,
\[|\mathcal{R}_W^{(2)}| =\frac{\lambda}{2}\left(|f|_2^2 - \Phi_W^\t \Sigma_W^{-1}\Phi_W\right) \leq \frac{2}{C_2}  + C(1 + \|\theta_\star\|_2^2) \lambda\epsilon_0^{2\kappa} \leq \frac{1}{C_2} + \frac{C n\epsilon_0^{2}}{2\sigma^2},\]
where in the ast inequality we have used (\ref{cond:lambda}). Hence, there exists an absolute constant $C$, such that for all $n$ large enough,
\begin{equation}\label{lem:lb:eq:2}
\int_{\mathcal{V}}e^{\mathcal{R}_W} \Pi_0(\rmd W) \geq C e^{-\frac{Cn\epsilon_0^2}{2\sigma^2}}\Pi_0(\mathcal{V}).
\end{equation}
Lemma \ref{lem:prior} below gives the lower bound
\[\Pi_0(\mathcal{V}) \geq \frac{1}{2}\exp\left(-s_0(\mathsf{u}+2)\log(q) -\frac{s_0}{2}(\|W_0\|_\infty +1)^2 -s_0\log\left(\frac{Ls_0^{1/2}}{\epsilon_0}\right)\right),\]
which together with assumption (\ref{tech:cond:3}) and the  inequality  in  (\ref{lem:lb:eq:2}) yields 
\[\int_{\mathcal{V}}e^{\mathcal{R}_W} \Pi_0(\rmd W) \geq c \times \exp\left(-C\left[\frac{n\epsilon_0^2}{\sigma^2}  +  s_0(\mathsf{u}+2)\log(q)\right]\right).\]
From the definition of $s_0$  in (\ref{eq:approx_s_val}), we have
\[s_0(\mathsf{u}+2)\log(q) \leq 6\mathsf{u}(s_0-1)\log(q) \leq 6\frac{n\epsilon_0^2}{\sigma^2},\]
which then yields (\ref{eq:lb:d}). The lemma follows from (\ref{step:1}) and (\ref{eq:lb:d}).
\end{proof}

We show here that the prior $\Pi_0$ has good contraction properties.
\begin{lemma}\label{lem:prior}
\begin{enumerate}
\item Given  $s,r\geq 0$, define
\[\W(s,r)\eqdef\left\{W\in \rset^q :\;\|W\|_0 \leq s,\;\mbox{ and }\; \|W\|_\infty \leq r\right\}.\]
If $r\geq \sqrt{(1+s)(2+\mathsf{u})\log(q)}$, we have
\[\Pi_0\left(\W(s,r)\right)\geq 1 - \frac{4}{q^{\mathsf{u}(1+s)}}.\]
\item Assume $q\geq \sqrt{2\pi}$. Fix $W_0\in\rset^q$ with sparsity support $\Lambda_0$. Given $r>0$, let 
\[\mathcal{V}(W_0,r) \eqdef \left\{W\in\rset^q:\; \mathsf{supp}(W)=\Lambda_0,\;\|W-W_0\|_\infty \leq r\right\}.\]
\[\Pi_0\left(\mathcal{V}(W_0,r)\right) \geq \frac{1}{2}\exp\left(-s_0(\mathsf{u}+2)\log(q) -\frac{s_0}{2}(\|W_0\|_\infty +r)^2  + s_0\log(r))\right),\]
where $s_0 = \|W_0\|_0$.
\end{enumerate}
\end{lemma}
\begin{proof}
\begin{enumerate}
\item  Let $\W^c$ be a short for $\rset^q\setminus W(s,r)$. Let $\bar\Pi_0$ denote the joint of $(\Lambda,W)$ in the definition of $\Pi_0$ (see Section \ref{sec:prior}). Then
\[\Pi_0(\W^c) = \Pi_0(\|\Lambda\|_0 >s) + \sum_{\Lambda:\;\|\Lambda\|_0\leq s} \bar\Pi_0(\Lambda)\times \bar\Pi_0(\|\Lambda\odot W\|_\infty >r \vert \Lambda).\]
 If $( \Lambda , W )\sim\bar\Pi_0$, then $ \Lambda $ is an ensemble of iid random variables drawn from the Bernoulli distribution with success probability $(1+ q^{\mathsf{u}+1})^{-1}$. Hence
\begin{multline*}
\Pi_0(\| \Lambda \|_0 > s) \leq \sum_{j> s} {q\choose j}\left(\frac{1}{1 + q^{\mathsf{u}+1}}\right)^{j} \left(\frac{q^{\mathsf{u}+1}}{1 + q^{\mathsf{u}+1}}\right)^{q - j} \\
\leq \sum_{j> s} {q\choose j} \left(\frac{1}{q^{\mathsf{u}+1}}\right)^j \leq 2\left(\frac{1}{q^{\mathsf{u}}}\right)^{s+1},\end{multline*}
where we use ${q\choose j}\leq q^j$, and $q^{\mathsf{u}}\geq 2$. 
Given $\Lambda_{d}=1$, $W_{d}\sim \mathbf{N}(0,1)$. Therefore, $\PP(|W_d|>t)\leq 2e^{-t^2/2}$ for all $t\geq 0$. Hence by union bound, for $\| \Lambda \|_0 \leq s$, and since $r\geq \sqrt{(1+s)(2+\mathsf{u})\log(q)}$, we obtain
\[
\Pi_0\left(\|\Lambda\odot W\|_\infty >r \; \vert \;  \Lambda \right) \leq
2e^{-r^2/2 + \log(s)} \leq \frac{2}{q^{\mathsf{u}(1+s)}}.\]
We conclude that
\begin{equation}\label{control:Pi:A}
\Pi_0(\W^c) \leq \frac{4}{q^{\mathsf{u}(1+s)}}. 
\end{equation}

\item  We write $\mathcal{V}$ as a short for $\mathcal{V}(W_0,r)$. We have
\[\Pi_0(\mathcal{V}) = \bar\Pi_0(\Lambda_0) \bar\Pi_0(\|\Lambda \odot W - W_0\|_\infty \leq r\vert \Lambda =\Lambda_0).\]
 Since $\log(1-x)\geq -2x$ for all $0\leq x\leq 1/2$, for $q^{\mathsf{u}} \geq 2/\log(2)$, we have
\begin{multline*}
\bar\Pi_0( \Lambda _0) = \left(\frac{1}{1+ q^{\mathsf{u}+1}}\right)^{\| \Lambda _{0}\|_0} \left(1 - \frac{1}{1+ q^{\mathsf{u}+1}}\right)^{q - \| \Lambda _{0}\|_0} \\
= \left(\frac{1}{q^{\mathsf{u}+1}}\right)^{\| \Lambda _{0}\|_0} \exp\left(q\log\left(1 - \frac{1}{1+ q^{\mathsf{u}+1}}\right)\right)\\
\geq \left(\frac{1}{q^{\mathsf{u}+1}}\right)^{\| \Lambda _{0}\|_0} \exp\left(-\frac{2q}{1+ q^{\mathsf{u}+1}}\right) \geq \frac{1}{2} \left(\frac{1}{q^{\mathsf{u}+1}}\right)^{\| \Lambda _{0}\|_0} = \frac{1}{2} \left(\frac{1}{q^{\mathsf{u}+1}}\right)^{s_0}.
\end{multline*}
If $U\sim \textbf{N}(0,1)$, and $t\geq 0$, then for all $a$, and $c\geq |a|$, \[P(|U-a|\leq t) \geq P(c \leq U\leq c + t) = \Phi(c+t) - \Phi(c)\geq e^{-(c+t)^2/2} \frac{t}{\sqrt{2\pi}},\]  
where $\Phi$ is the cdf of the standard normal distribution. We use this inequality with $c = \|W_0\|_\infty$, and we the assumption $q \geq \sqrt{2\pi}$, we deduce that
\begin{multline*}
\bar\Pi_0\left(\|\Lambda \odot W - W_0\|_\infty \leq r\vert \Lambda =\Lambda_0\right) \geq \frac{1}{q^{s_0}}\exp\left(-\frac{s_0}{2}(\|W_0\|_\infty +r)^2 + s_0 \log(r)\right).
\end{multline*}
Hence
\[\bar\Pi_0( \Lambda _0) \geq \frac{1}{2}\left(\frac{1}{q^{\mathsf{u}+2}}\right)^{s_0} \exp\left(-\frac{s_0}{2}(\|W_0\|_\infty +r)^2 + s_0 \log(r)\right),\]
as claimed.
\end{enumerate}
\end{proof}

\subsection{Ignorability of unsuitable weights}\label{sec:prior:concentration}

Given integer $s\geq 0$, $\tau_s\eqdef \sqrt{(2+\mathsf{u})(1+s)\log(q)}$, and a constant $C$, we define that 
\[\W_0(s,C) \eqdef \left\{W\in\rset^q:\;\|W\|_0\leq s,\;\mbox{ and }\; \|W\|_\infty \leq \tau_s,\;\mbox{ and }\; {\cal J}(u_W)\leq \frac{C n\epsilon_0^2}{\lambda\sigma^2}\right\}.\]

Our next result shows that the prior $\Pi_0$ puts most of its probability mass on $\W_0(s,C)$.
 
\begin{lemma}\label{lem:good:set}
Assume H\ref{H1}-H\ref{H:stability}. Let $s_0, \epsilon_0$, $W_0$ be as in H\ref{H:approx}. We can find an absolute constant $C$ such that with
\begin{equation}
\label{choice:s}
 s = C s_0,
 \end{equation}
it holds,
\[\PE\left[\Pi\left(\W_0(s,C)\vert \D\right)\right] \geq 1 - C_0 e^{-\frac{n\epsilon_0^2}{2\sigma^2}},\]
for some absolute constant $C_0$. 
\end{lemma}
\begin{proof}
For any measurable set $A\subseteq\rset^q$, we have
\[\Pi(A\vert \D) = \frac{\int_{A} \exp\left(-\frac{n}{2\sigma^2}\PP_n(\Delta_W-\Delta_\star) + \mathcal{R}_W\right)\Pi_0(\rmd W)}{\int_{\rset^q} \exp\left(-\frac{n}{2\sigma^2}\PP_n(\Delta_W-\Delta_\star) + \mathcal{R}_W\right)\Pi_0(\rmd W)}.\]
By Lemma \ref{lem:lb}, we can find an absolute constant $C_0$, such that for all $n$ large enough, and with 
\[\beta \eqdef  \frac{n\epsilon_0^2}{2\sigma^2},\]
\begin{multline*}
\PE\left[\Pi\left(A\vert \D\right)\right] \leq 2e^{-n\epsilon_0^2/(2\sigma^2)} +   \frac{e^{C0\beta}}{C_0}\PE\left[\int_{A} \exp\left(-\frac{n}{2\sigma^2}\PP_n(\Delta_W-\Delta_\star) + \mathcal{R}_W\right)\Pi_0(\rmd W)\right].\end{multline*}
By Fubini's theorem, the expectation on the right hand side of the last display is
\[\int_{A} e^{\mathcal{R}_W} \PE\left[\exp\left(-\frac{n}{2\sigma^2}\PP_n(\Delta_W-\Delta_\star)\right)\right]\Pi_0(\rmd W).\]
By conditioning on ${\bf s}_{1:n}$ we see that the expectation inside the last integral is equal to $1$ for all $W$. Hence,
\begin{equation}\label{lem:good:set:eq1}
\PE\left[\Pi\left(A\vert \D\right)\right] \leq 2e^{-n\epsilon_0^2/(2\sigma^2)}  + \frac{e^{C_0\beta}}{C_0}\;\int_{A} e^{\mathcal{R}_W}\Pi_0(\rmd W).\end{equation}
By definition, we have $\mathcal{R}_W = -\frac{1}{2}\left(\log\det(\Sigma_W) - \log\det(\Sigma_\star)\right) -\lambda {\cal J}(u_W)/2$. For all $W\in\rset^q$, we have
\begin{equation}\label{bound:logdet}
-\frac{1}{2}\left(\log\det(\Sigma_W) - \log\det(\Sigma_\star)\right) \leq \frac{d}{2}\log\left(n\|\Sigma_\star\|_{\mathsf{op}}\right)\leq   C_1\beta,\end{equation}
where the second inequality uses (\ref{tech:cond:1}). We note also that ${\cal J}(u_W)\geq 0$. As a result (\ref{lem:good:set:eq1}) becomes 
\begin{equation}\label{lem:good:set:eq2}
\PE\left[\Pi\left(A\vert \D\right)\right] \leq 2e^{-n\epsilon_0^2/(2\sigma^2)}  + \frac{e^{C_0 \beta}}{C_0}\;\int_{A} e^{-\frac{\lambda}{2}{\cal J}(u_W)}\Pi_0(\rmd W),\end{equation}
for some possibly different absolute constant $C_0$. We apply this with $A = \W_0(Cs_0,C)^c = \{W\in\rset^q:\;{\cal J}(u_W)> 2C\beta/\lambda\}\cup[\W_0(s)^c\cap \{W\in\rset^q:\;{\cal J}(u_W)\leq  2C\beta/\lambda\}]$. For $C\geq 1+C_0$,
\[\PE\left[\Pi\left(\W(s)^c\vert \D\right)\right]
\leq 2e^{-n\epsilon_0^2/(2\sigma^2)} + \frac{1}{C_0}e^{-n\epsilon_0^2/(2\sigma^2)} + \frac{e^{C_0 \beta}}{C_0}\;\int_{\W(s)^c} \Pi_0(\rmd W),\]
where $\W(s) = \{W\in\rset^q:\;\|W\|_0\leq s,\;\mbox{ and }\; \|W\|_\infty \leq \tau_s\}$. Lemma \ref{lem:prior}-(1) shows that $\Pi_0(\W(s)^c) \leq 4\exp(-\mathsf{u}s\log(q))$. As a result, by taking  $s = C s_0$ with $C \geq C_0 +1$, we have 
\[\mathsf{u}s\log(q) = C\mathsf{u}s_0\log(q) \geq \frac{C n\epsilon_0^2}{2\sigma^2} \geq \frac{(C_0+1)n\epsilon_0^2}{2\sigma^2},\]
and we conclude that
\[\PE\left[\Pi\left(\W_0(s)^c\vert \D\right)\right]
\leq 2e^{-n\epsilon_0^2/(2\sigma^2)} + \frac{1}{C_0}e^{-n\epsilon_0^2/(2\sigma^2)} \leq C \times e^{-n\epsilon_0^2/(2\sigma^2)},\]
for some absolute constant $C$. 

\end{proof}

\subsection{Deviation bounds}\label{sec:emp:proc}
The last ingredient of the proof is a concentration inequality for the empirical process of the log-likelihood ratio that we derive next. The proof is based on Theorem 5.11 of \cite{geer2000empirical} that we first present. For any random variable $X$, its $\kappa$-Bernstein norm $\rho_\kappa$ is defined as: 
\begin{equation}
\label{eq:bernstein_norm}
\rho_\kappa(X) \eqdef \sqrt{2\kappa^2 \bbE\left[e^{\frac{|X|}{\kappa}} - 1 - \frac{|X|}{\kappa}\right]} \,.
\end{equation}
We recall that we say a random variable satisfies Bernstein condition with parameter $(K, R)$ if: 
$$
\bbE[|X|^m] \le \frac12 m! K^{m-2} R^2 \ \ \forall \ \ m = 2, 3, \dots 
$$
It is immediate that if a random variable $X$ satisfies Bernstein condition with parameters $(K, R)$ then $\rho_{2K}(X) \le \sqrt{2}R$. Let $X,X_{1:n}\stackrel{i.i.d.}{\sim} P$, $\PP_n$ their empirical measure.  Let $\cG$ be a collection of real-valued functions, where $g(X)$ satisfies the Bernstein condition with parameters $(K, R)$ for all $g \in \cG$. Let $H_{B, 2K}(u, \cG, P)$ denote the bracketing entropy of $\cG$ with respect to the pseudo-metric $\rho_{2K}$.

\begin{theorem}[Theorem 5.11 of \cite{geer2000empirical}]
\label{thm:vdg}
Suppose that $g(X)$ satisfies the Bernstein condition with parameters $(K, R)$ for all $g \in \cG$. Then there  exists a universal constant $C$ such that for any $a,C_0,C_1$ satisfying: 
\begin{equation}\label{thm:vdg:cond}
C_0\left(\max\left\{\int_0^{R}\sqrt{H_{B, 2K}(u, \cG, P)} \ du, \ R \right\}\right) \le a \le \min\left\{\frac{C_1 \sqrt{n}R^2}{2K}, 8 \sqrt{n}R\right\}
\end{equation}
and $C^2 \le C_0^2/(C_1 + 1)$: 
$$
\bbP\left(\sup_{g \in \cG}\left|\sqrt{n}(\bbP_n - P)g\right| \ge a\right) \le Ce^{- \frac{a^2}{C^2 (C_1 + 1)R^2}} \,.
$$
\end{theorem}

Given $s\geq 0$, we recall the definition  
\[\W(s)=\left\{W\in\W:\;\|W\|_0\leq s,\;\; \|W\|_\infty\leq \tau_s\right\},\;\mbox{ where }\; \tau_s = \sqrt{(2+\mathsf{u})(1+s)\log(q)}.\]
We also recall that
\[\Delta_W(y,{\bf s}) \eqdef \left(y-u_W({\bf s})\right)^2,\;\; \Delta_\star(y,{\bf s}) \eqdef \left(y-u_\star({\bf s})\right)^2,\;\;(y,{\bf s})\in\rset\times\Omega.\]

We use Theorem \ref{thm:vdg} to obtain the following.

\begin{lemma}\label{lem:dev:bound}
Assume H\ref{H1}-H\ref{H:lip}. Fix $s\geq 1$,  and let $\F_s$, $V_1=V_1(s)$ and $V_2=V_2(s)\geq 6b$ be as in H\ref{H:lip}. There exists an absolute constant $M_0\geq 1$ such that the following holds. For all $M\geq M_0$, $A_M\eqdef Mb\sqrt{V_1\log(V_2\sqrt{n})}$, and $\delta>0$ such that  
\[\frac{2\sqrt{2}}{\sqrt{n}}\leq \frac{A_M}{\sqrt{n}} \leq \delta \leq 2b,\]
we have
\[\PP\left[\sup_{W\in\W(s):\; |u_W - u_\star|_2\leq \delta}\; \;\left|\sqrt{n}(\PP_n-\PP)(\Delta_W - \Delta_\star)\right| > \frac{A_M \delta}{2}\left(1 + \frac{\sigma}{b}\right)\right] \leq C e^{-\frac{A_M^2}{C b^2}},\]
for some absolute constant $C$.
\end{lemma}
\begin{proof}
We set
$$
Z_n(\delta) \eqdef \sup_{W\in\W(s):\; |u_W - u_\star|_2\leq \delta}\;\; \left|\sqrt{n}(\PP_n-\PP)(\Delta_W - \Delta_\star)\right| \,.
$$
Recall that, by definition of $Y$, we have $Y_i  = u_\star({\bf s}_i) + \sigma \xi_i$ where $\xi_i \sim \mathbb{N}(0, 1)$, $(\xi_i,{\bf s}_i)$ are independent, and $(\xi_i,{\bf s}_i)\sim\PP$. Therefore, we can simplify the difference $(\Delta_W - \Delta_\star)(\xi,{\bf s})$ as: 
$$
(\Delta_W - \Delta_\star)(\xi,{\bf s}) = (u_\star({\bf s}) - u_W({\bf s}))^2 + 2\sigma \xi(u_\star({\bf s}) - u_W({\bf s})) \,,
$$
and consequently, we can write: 
\begin{align*}
Z_n(\delta) & \le \sup_{W\in\W(s):\; |u_W - u_\star|_2\leq \delta}\; \left|\sqrt{n}(\PP_n-\PP)(u_\star({\bf s}) - u_W({\bf s}))^2 \right| \\
& \qquad \qquad +  2\sigma \sup_{W\in\W(s):\; |u_W - u_\star|_2\leq \delta}\; \left|\sqrt{n}(\PP_n-\PP)(\xi u_\star({\bf s}) - \xi u_W({\bf s}))\right| \\
& = Z_{n, 1}(\delta) + 2\sigma Z_{n, 2}(\delta) \, ,
\end{align*}
where $Z_{n, 1}(\delta)$ and $Z_{n, 2}(\delta) $ can also be written as 
$$Z_{n,1}(\delta) \eqdef \sup_{h\in\mathcal{H}_\delta}\;\left|\frac{1}{\sqrt{n}}\sum_{i=1}^n(h({\bf s}_i) - \PP(h))\right|,\;\;\;\mbox{ and }\;\;\; Z_{n, 2}(\delta) \eqdef  \sup_{f \in \bar\cF_{\delta}}\left|\frac{1}{\sqrt{n}}\sum_{i = 1}^n f(\xi_i,X_i)\right|,$$
where $\mathcal{H}_\delta\eqdef \{f^2,\;f\in\cF_\delta\}$, and where $\cF_{\delta} = \{u_W - u_\star: W\in\W(s),\; |u_W - u_\star|_2 \le \delta\}$, and $\bar\cF_{\delta}=\{(\xi,x)\mapsto \xi f(x),\;f\in\cF_\delta\}$.

\paragraph{\texttt{(A) Deviation bound for $Z_{n, 1}(\delta)$}}\; 
By H\ref{H:lip}-(1), $|h|_\infty \leq b^2$ for all $h\in\mathcal{H}_\delta$. Therefore for $f\in\cF_\delta$, and $m\geq 2$, we have: 
$$
\bbE[|f^2(X)|^m] \le (b^2)^{m-2}\bbE[|f^2(X)|^2] \le (b^2)^{m-2}(b\delta)^2 \le \frac12 m!  (b^2)^{m-2}(b\delta)^2 \,.
$$
Therefore the function class $\mathcal{H}_{\delta}$ satisfies the Bernstein condition with parameter $(b^2, b\delta)$. The bracketing number of $\cH_{\delta}$ satisfies 
\begin{equation}\label{bound:bracket:2}
\cH_{B, 2b^2}(u\sqrt{2}, \cH_{\delta}, \PP) \le \log\mathcal{N}(u/2b, \cF_{\delta}, L_\infty) \,,
\end{equation}
for all $0<u<b^2$. To see this fix some $u > 0$ and let $\{f_j\}_{1 \le j \le N}$ be $(u/2b)$-cover of $\cF_{\delta}$ with respect to $L_\infty$ norm. Therefore, for any $f \in \cF_{\delta}$ there exists $1 \le j \le N$ such that $\|f - f_j\|_\infty \le u/2b$. Consequently, we have $\|f^2 - f_j^2\|_\infty \le 2b (u/2b) = u$. Now consider the collection of brackets $\{(f_j^2 - u/2, f_j^2 + u/2)\}_{1 \le j \le N}$.  For this collection we have, $\rho_{2b^2}^2 \left(f_j^2 + \frac{u}{2} - f_j^2 + \frac{u}{2}\right)  = \rho_{2b^2}^2(u) \leq (\sqrt{2}u)^2$, for $u\leq b^2$. Hence we have established (\ref{bound:bracket:2}). 

We can then apply Theorem \ref{thm:vdg} with  $a= A_M\delta/2$, $R=b\delta$, $K=b^2$, $C_1=1$, $C_0^2=2C$, where $C$ is the absolute constant in Theorem \ref{thm:vdg}. The condition $a\geq C_0 R$ is satisfied with $M\geq M_0\geq 2C_0$. For $\delta\leq 2b$, using H\ref{H:lip}-(4),
\begin{align*}
\int_0^{b\delta}\sqrt{\cH_{B, 2b^2}(u, \cH_{\delta}, \PP)}\rmd u & = \sqrt{2}\int_0^{b\delta/\sqrt{2}}\sqrt{\cH_{B, 2b^2}(u\sqrt{2}, \cH_{\delta}, \PP)} \ du \\
& \le  \sqrt{2} \int_0^{b\delta/ \sqrt{2}}\sqrt{\log{\cN(u/2b, \cF_\delta, L_\infty)}} \ du  \\
& \le 2b\sqrt{2V_1} \int_0^{\delta/2\sqrt{2}} \sqrt{\log{\left(\frac{V_2}{u}\right)}} \ du \\
& \le 2b\delta\sqrt{V_1}  \sqrt{\log{\left(\frac{2\sqrt{2}V_2}{\delta}\right)}}  \leq  2b\delta  \sqrt{V_1\log{\left(V_2\sqrt{n} \right)}},
\end{align*}
where the inequality before last follows from the following fact: 
$$
\int_0^c \sqrt{\log{\left(\frac{C}{u}\right)}} \ du \le 2c\sqrt{\log{\left(\frac{C}{c}\right)}} \ \ \text{as soon as } \log{(C/c)} \ge 2 \,,
$$
which in our case is implied by $V_2\geq 6b$. The last inequality uses $\delta\geq 2\sqrt{2}/\sqrt{n}$.

Therefore the condition $a \geq C_0 \int_0^{b\delta}\sqrt{\cH_{B, 2b^2}(u, \cH_{\delta}, \PP)}\rmd u$ is satisfied with $M \geq M_0 \geq 4C_0$. The condition $a \leq C_1\sqrt{n}R^2/(2K)$ boils down to $A_M/\sqrt{n}\leq \delta$, whereas the condition $a\leq 8R\sqrt{n}$ boils down to $A_M/\sqrt{n} \leq 16b$ which holds by assumption. Hence by Theorem \ref{thm:vdg},
\[\PP\left(Z_{n,1}(\delta)> \frac{A_M\delta}{2}\right) \leq Ce^{-\frac{A_M^2}{C b^2}},\]
for some absolute constant $C$.

\paragraph{\texttt{(B) Deviation bound for $Z_{n, 2}(\delta)$}}\; The argument is similar.  Note that, for any $u_W - u_\star \in \cF_{\delta}$, $|u_W -u_\star|_2 \leq  \delta$, and since the function $u_W -u_\star$ are bounded by $b$ as assumed in H\ref{H:lip}, we have:
\begin{align*}
\bbE[|\xi f(X)|^m] & = \bbE[|\xi|^m]\bbE[|u_W(X) - u_\star(X)|^m] \\
& \le \bbE[|\xi|^m] b^{m-2} \delta^2\\
& \le \frac12 m! b^{m-2} \delta^2 \,,
\end{align*}
where the last inequality follows from the centered absolute moment of Gaussian random variables. Therefore, all $f\in\bar\F_\delta$ satisfies the Bernstein condition with parameters $(b, \delta)$. 

The bracketing entropy satisfies,
\begin{equation}
\label{bound:bracket:1}
\cH_{B, 2b}(u\sqrt{2}, \bar\cF_{\delta}, \PP) \le \log{\cN(u/2, \cF_{\delta}, L_\infty)} \,\;\; 0 < u \le \delta\,.
\end{equation}
To establish (\ref{bound:bracket:1}), fix $u\in (0,b]$, and let $\{f_1, \dots, f_N \}$ be a $u/2$ cover of $\cF_{\delta}$ with respect to $L_\infty$ norm, i.e. 
$$
\sup_{f \in \cF_\delta} \min_{1 \le j \le N} \|f  - f_j\|_\infty \le \frac{u}{2} \,.
$$ 
Now consider the brackets $\{(\xi,x)\mapsto (\xi f_j(x) -u|\xi|/2, \xi f_j(x) +u|\xi|/2),\;1\leq j\leq N\}$.  For any $f$, there exists $f_j$ such that $\|f - f_j\|_\infty \le u/2$ due to the property of the covering set, and it follows that 
$$
\xi f_j(x) - \frac{u |\xi|}{2} \le \xi f(x) \le \xi f_j(x) + \frac{u |\xi|}{2} \ \ \forall \ \ x \,.
$$
Furthermore, for any $1 \le j \le N$, since $\PE(|u\xi|^m)\leq u^mm!/2\leq b^{m-2}u^2m!/2$, 
\[ \rho_{2b}^2 \left(\xi f_j(x) +u|\xi|/2 - \xi f_j(x)  + u|\xi|/2\right) = \rho_{2b}^2 \left(u |\xi|\right) \leq (\sqrt{2}u)^2,\]
which implies (\ref{bound:bracket:1}).

Hence we can apply Theorem \ref{thm:vdg} with $a=A_M\delta/2b$, $R=\delta$, $K=b$, $C_1=1$, $C_0^2=2C$, where $C$ is the absolute constant in Theorem \ref{thm:vdg}. All the conditions of Theorem \ref{thm:vdg} can be checked as we did in bounding $Z_{n,1}(\delta)$. 
We conclude that
\[\PP\left(Z_{n,2}(\delta)> \frac{A_M \delta}{2b}\right) \leq Ce^{-\frac{A_M^2}{C b^2}},\]
for some absolute constant $C$. Combining the two bounds, we get
\[\PP\left(Z_{n}(\delta) > \frac{A_M \delta}{2}\left(1 + \frac{\sigma}{b}\right)\right) \leq Ce^{-\frac{A_M^2}{C b^2}},\]
for some absolute constant $C$, which implies the stated result.
\end{proof}

\subsection{Finishing the proof}\label{sec:finish:proof:thm:margW}

\begin{proof}
Let $s_0,\epsilon_0$, and $W_0$ be as in H\ref{H:approx}. Let $s=C s_0$ as in (\ref{choice:s}).  We set $\tau_s=\sqrt{(2+\mathsf{u})(1+s)\log(q)}$, 
\[\W_0(s) \eqdef \left\{W\in\rset^q:\;\|W\|_0\leq s,\;\|W\|_\infty \leq \tau_s,\;{\cal J}(u_W) \leq \frac{C n\epsilon_0^2}{\lambda\sigma^2}\right\}.\]

We set  $\r \eqdef 2M_0(b+\sigma)\sqrt{V_1\log(V_2\sqrt{n})/n}$, where $M_0$ is as in Lemma \ref{lem:dev:bound}.  For $j\geq 1$, we set $A_j \eqdef bM_0 j\sqrt{V_1\log(V_2\sqrt{n})}$, $\r_j \eqdef j\r$. With $C$ as in Lemma \ref{lem:lb}, we set  
\[\beta \eqdef \frac{C}{\sigma^2}n\epsilon_0^2,\]
and define 
\begin{multline*}
\e\eqdef\left\{\D:\;\int_{\rset^q} \exp\left(-\frac{n}{2\sigma^2}\PP_n(\Delta_W-\Delta_\star) + \mathcal{R}_W\right)\Pi_0(\rmd W)
\leq C e^{-\beta},\;\;\mbox{ or }\; \right.\\
\left. \sup_{W\in\W(s):\; \r_j<|u_W-u_\star|_2\leq\r_{j+1}}\; \left|\sqrt{n}(\PP_n-\PP)(\Delta_W - \Delta_\star)\right| > \frac{A_{j+1} \r_{j+1}}{2}\left(1 + \frac{\sigma}{b}\right) \;\;\mbox{ for some }\; j\geq 1\right\}.
\end{multline*}
By Lemma \ref{lem:lb} and Lemma \ref{lem:dev:bound}, and for all $n$ large enough
\begin{equation}\label{prob:bad:set}
\PP(\D\in \e) \leq 2e^{-n\epsilon_0^2/(2\sigma^2)} + C\sum_{j\geq 1} e^{-A_j^2/(Cb^2)} \leq 2e^{-n\epsilon_0^2/(2\sigma^2)} +Ce^{-V_1\log(V_2\sqrt{n})/C},\end{equation}
for some absolute constant $C$.

We set $B\eqdef\{W\in\rset^q:\; |u_W-u_\star|_2 >  \r\}$. We can write  $B\cap\W(s)$ as $\cup_{j\geq 1} B_j$, where 
\[B_j\eqdef\left\{W\in\W(s):\;M j \r < |u_W-u_\star|_2 \leq M(j+1)\r\right\}.\]
We should point out that the union $\cup_{j\geq 1} B_j$ is over a finite number of terms since $B_j = \emptyset$ for $Mj\r\geq b$.  Since $\Pi(B\vert\D) \leq \Pi(\W_0(s)^c\vert\D) + \textbf{1}_\e(\D) + \textbf{1}_{\e^c}(\D)\Pi(B\cap\W_0(s)\vert\D)$, taking expectation on both sides, and using Lemma \ref{lem:good:set}, the definition of $\e$ and  (\ref{prob:bad:set}), yields for all $n$ large enough,
\begin{multline*}
\PE\left[\Pi(B\vert\D)\right] \leq C e^{-n\epsilon_0^2/(2\sigma^2)} + Ce^{-V_1\log(V_2\sqrt{n})/C} \\
+ \frac{e^{\beta}}{C}\sum_{j\geq 1} \PE\left[\textbf{1}_{\e^c}(\D) \;\int_{B_j} \exp\left(-\frac{n}{2\sigma^2}\PP_n(\Delta_W-\Delta_\star) + \mathcal{R}_W\right)\Pi_0(\rmd W)\right],
\end{multline*}
for some absolute constant $C$. We have
\begin{multline*}
\frac{n}{2\sigma^2}\PP_n(\Delta_W-\Delta_\star) + \mathcal{R}_W = -\frac{n}{2\sigma^2}\varrho(u_W,u_\star) -\frac{n}{2\sigma^2}(\mathbb{P}_n-\mathbb{P})(\Delta_W - \Delta_\star) + \mathcal{R}_W',
\end{multline*}
where 
\[\mathcal{R}_W' = -\frac{1}{2}\left(\log\det\Sigma_\star - \log\det\Sigma_W\right),\]
and satisfies, as we show in the proof of Lemma \ref{lem:good:set}
\[\left|\mathcal{R}_W'\right| \leq  \frac{d}{2}\log\left(\frac{n\|\Sigma_\star\|_{\textsf{op}}}{\sigma^2}\right) \leq C_0\log(q),\]
where the second inequality is our assumption (\ref{tech:cond:1}). Hence, we can find an absolute constant $C$ such that for all $n$ large enough
\begin{multline}\label{proof:thm:margW:eq1}
\PE\left[\Pi(B\vert\D)\right] \leq C e^{-n\epsilon_0^2/(2\sigma^2)} + Ce^{-V_1\log(V_2\sqrt{n})/C} \\
+ \frac{e^{\beta}}{C}\sum_{j\geq 1} \PE\left[\textbf{1}_{\e^c}(\D) \;\int_{B_j} \exp\left(-\frac{n}{2\sigma^2}\varrho(u_W,u_\star) - \frac{n}{2\sigma^2}(\PP_n-\PP)(\Delta_W-\Delta_\star)\right)\Pi_0(\rmd W)\right].
\end{multline}
For $W\in B_j$, and $\D\notin\e$, the expression inside the exponential in the last display is bounded from above by
\[-\frac{n\r_j^2}{2\sigma^2} + \frac{\sqrt{n}A_{j+1} \r_{j+1}}{2\sigma^2}\left(1+ \frac{\sigma}{b}\right) \leq -\frac{n\r_j^2}{4\sigma^2},\]
with the choice of $\r$. Since
\[\sum_{j\geq 1} e^{-\frac{n\r_j^2}{4\sigma^2}} \leq C e^{-\frac{n\r^2}{4\sigma^2}},\]
for some absolute constant $C$, we conclude that for all $n$ large enough,
\[ \PE\left[\Pi(B\vert\D)\right] \leq C \left(e^{-n\epsilon_0^2/(2\sigma^2)}  + e^{-V_1\log(V_2\sqrt{n})/C} + e^{-n\r^2/(4\sigma^2)}\right).\]
Hence the theorem.
\end{proof}

\subsection{Proof of Theorem \ref{thm:marg:theta}}\label{sec:proof:thm:marg:theta}
\begin{proof}
First, a simple coupling argument shows that
\begin{equation}\label{thm:marg:theta:eq1}
\mathsf{W}_2^2(\Pi^{(\theta)}, \Pi^{(\theta)}_\star) \leq \int_{\rset^q} \mathsf{W}_2^2\left(\textbf{N}(\hat\theta_W,\frac{1}{\lambda}\Sigma_W^{-1}),\; \textbf{N}(\theta_\star,\frac{1}{\lambda}\Sigma_\star^{-1})\right) \Pi^{(W)}(\rmd W\vert\D).
\end{equation}
We recall (see e.g. \cite{bhatia:etal:19}) that for symmetric and positive definite matrices $\Lambda_1,\Lambda_2$,
\begin{multline*}
\mathsf{W}_2^2\left(\textbf{N}(\mu_1,\Lambda_1),\textbf{N}(\mu_2,\Lambda_2)\right) = \|\mu_2-\mu_1\|_2^2 +\textsf{Tr}\left(\Lambda_1 + \Lambda_2 -2\left(\Lambda_1^{1/2} \Lambda_2\Lambda_1^{1/2}\right)^{1/2}\right)\\
\leq \|\mu_2-\mu_1\|_2^2 + \|\Lambda_1^{1/2} - \Lambda_2^{1/2}\|_{\textsf{F}}^2,
\end{multline*}
where the inequality follows from Theorem 1 of (\cite{bhatia:etal:19}). By the Hemmen-Ando inequality (\cite{ando:hemmen:80}~Proposition 2.1),
\[\|\Lambda_1^{1/2} - \Lambda_2^{1/2}\|_{\textsf{F}} \leq \frac{\|\Lambda_1 - \Lambda_2\|_{\textsf{F}}}{\sqrt{\lambda_{\textsf{min}}(\Lambda_1)} +\sqrt{\lambda_{\textsf{min}}(\Lambda_2)}}.\]
Hence
\begin{equation}\label{W2:gaussian}
\mathsf{W}_2^2\left(\textbf{N}(\mu_1,\Lambda_1),\textbf{N}(\mu_2,\Lambda_2)\right) \leq \|\mu_2-\mu_1\|_2^2 + \frac{\|\Lambda_1 - \Lambda_2\|_{\textsf{F}}^2}{\left(\sqrt{\lambda_{\textsf{min}}(\Lambda_1)} +\sqrt{\lambda_{\textsf{min}}(\Lambda_2)}\right)^2}.\end{equation}
We apply this bound, to conclude that there exists a constant $c$ (that we can take as $d/\lambda_{\textsf{max}}(\Sigma_\star)$ such that for all $W\in\rset^q$,
\[\mathsf{W}_2^2\left(\textbf{N}(\hat\theta_W,\frac{1}{\lambda}\Sigma_W^{-1}),\; \textbf{N}(\theta_\star,\frac{1}{\lambda}\Sigma_\star^{-1})\right) \leq \|\hat\theta_W -\theta_\star\|_2^2 + \frac{c}{\lambda}\|\Sigma_W^{-1}-\Sigma_\star^{-1}\|^2_{\textsf{op}}.\]

We consider the integrand of (\ref{thm:marg:theta:eq1}) under two scenarios.

\paragraph{\underline{\texttt{Case 1: $W$ is such that $|u_W-u_\star|_2\leq M\r$}}} By (\ref{comp:Sigma}), and for $W$ such that $|u_W-u_\star|_2\leq M\r$, for some absolute onstant $C$ we have
\[\|\Sigma_W - \Sigma_\star\|_{\textsf{op}} \leq C\left(|\H u_\star|_2 + (M\r)^\kappa\right) (M\r)^\kappa,\]
for all $n$ large enough. So, since $\lambda_{\textsf{min}}(\Sigma_\star)>C_2$, we conclude that for all such $W$ under consideration, and for all $n$ large enough, $\lambda_{\textsf{min}}(\Sigma_W)>C_2/2$. 

For all $W$, 
\begin{multline*}
(\hat\theta_W - \theta_\star)^\t \H_1 u_W = \hat\theta_W^\t \H_1 u_W - \theta_\star^\t \H_1 u_\star - \theta_\star^\t(\H_1 u_W - \H_1 u_\star) \\
= \hat\theta_W^\t \H_1 u_W  -\H_0 u_W -f + (\H_0u_W - \H_0 u_\star) - \theta_\star^\t(\H_1 u_W - \H_1 u_\star).
\end{multline*}
Therefore
\[|(\hat\theta_W - \theta_\star)^\t \H_1 u_W|_2^2 \leq 2|f - \H_0 u_W -  \hat\theta_W^\t \H_1 u_W|_2^2 + c \times |\H u_W - \H u_\star|_2^2.\]
Since $\hat\theta_W$ minimizes $\theta\mapsto |f - \H_0 u_W-\theta^\t\H_1 u_W|_2 + \|\theta\|_2^2/\lambda$, we have
\begin{multline*}
|f - \H_0 u_W -  \hat\theta_W^\t \H_1 u_W|_2^2 \leq |f - \H_0 u_W -  \theta_\star^\t \H_1 u_W|_2^2 + \frac{\|\theta_\star\|_2^2}{\lambda} \\
\leq c \times |\H u_W - \H u_\star|_2^2 + \frac{\|\theta_\star\|_2^2}{\lambda}.
\end{multline*}
Using the above, and H\ref{H:stability}, we conclude that we can find a constant $c$ such that for $W$ such that  $|u_W - u_\star|_2\leq M\r$, it holds
\[|(\hat\theta_W - \theta_\star)^\t \H_1 u_W|_2^2 \leq c \r^{2\kappa}.\]
However, $|(\hat\theta_W - \theta_\star)^\t \H_1 u_W|_2^2 = (\hat\theta_W - \theta_\star)^\t (\Sigma_W - (1/\lambda)\textbf{1}_d)(\hat\theta_W - \theta_\star)$. And since $\lambda_{\textsf{min}}(\Sigma_W)>C/2$ as seen above, and since $\lambda\to\infty$, as $n\to\infty$, for all $n$ large enough we have
\[|(\hat\theta_W - \theta_\star)^\t \H_1 u_W|_2^2 \geq C \|\hat\theta_W - \theta_\star\|_2^2,\]
for some absolute constant $C$. In conclusion, for $W$ such that  $|u_W - u_\star|_2\leq M\r$, and for all $n$ large enough,
\[\| \hat\theta_W - \theta_\star\|_2^2 \leq c\r^{2\kappa},\]
and
\[\mathsf{W}_2^2\left(\textbf{N}(\hat\theta_W,\frac{1}{\lambda}\Sigma_W^{-1}),\; \textbf{N}(\theta_\star,\frac{1}{\lambda}\Sigma_\star^{-1})\right)\leq c\r^{2\kappa}.\]

\paragraph{\underline{\texttt{Case 2: $W$ is such that $|u_W - u_\star|_2 > M\r_n$}}}
Since the smallest eigenvalue of $\Sigma_W$ is at least $1/\lambda$, there exists a constant $c$ such that for all $W\in\rset^q$, we have
\[\|\Sigma_W^{-1}-\Sigma_\star^{-1}\|_{\textsf{op}} \leq c\lambda.\]
Similarly,
\[\|\hat\theta_W - \theta_\star\|_2 \leq \|\hat\theta_W\|_2 + \|\theta_\star\|_2,\]
and
\[\|\hat\theta_W\|_2^2 = \hat\theta_W^\t\Sigma_W^{1/2}\Sigma_W^{-1}\Sigma_W^{1/2}\hat\theta_W \leq \lambda \hat\theta_W^\t \Sigma_W \hat\theta_W = \lambda \Phi_W^\t \Sigma_W^{-1} \Phi_W \leq \lambda |\bar f_W|_2^2\leq 2\lambda(|f|_2^2 + |\H_0 u_W|_2^2).\]
We deduce that for all $W$ such that $|u_W - u_\star|_2 > M\r$, we have
\[\mathsf{W}_2^2\left(\textbf{N}(\hat\theta_W,\frac{\sigma^2}{\lambda}\Sigma_W^{-1}),\; \textbf{N}(\theta_\star,\frac{\sigma^2}{\lambda}\Sigma_\star^{-1})\right)\leq c \lambda\left(1 +|\H_0 u_W|_2^2\right),\]
for some constant $c$. Given the behaviors of the $\mathsf{W}_2$ distance between $\textbf{N}(\hat\theta_W,\frac{\1}{\lambda}\Sigma_W^{-1})$ and $\textbf{N}(\theta_\star,\frac{1}{\lambda}\Sigma_\star^{-1})$  obtained  in the cases above, we return to (\ref{thm:marg:theta:eq1}) to write
\[\mathsf{W}_2^2(\Pi^{(\theta)}, \Pi^{(\theta)}_\star) \leq c\r^{2\kappa} +  c\lambda \int_{\rset^q}\left(1 + |\H_0 u_W|_2^2\right)\textbf{1}_{\{|u_W - u_\star|_2 > M\r_n\}} \Pi^{(W)}(\rmd W\vert \D).\]
Taking the expectation on both sides and with a similar argument as in the proof of Theorem \ref{thm:margW} yields
\[\PE\left[\mathsf{W}_2^2(\Pi^{(\theta)}, \Pi^{(\theta)}_\star)\right] \leq c\left(\r^{2\kappa} + \lambda e^{-n\epsilon_0^2/(2\sigma^2)}  + \lambda e^{-V_1\log(V_2\sqrt{n})/C}\right),\]
which is the stated bound.
\end{proof}

\subsection{Proof of Theorem \ref{thm:approx:deriv}}\label{sec:proof:thm:approx:deriv}
\begin{proof}
The result is based on a generalization of Markov's polynomial inequality due to \cite{harris:08}~Theorem 1. The result shows that differential operators are bounded operators when restricted to polynomials, and an upper on their norms is given. 

Let $X,Y$ be a real Banach space with norms $\|\cdot\|_X$, and $\|\cdot\|_Y$ respectively. A function $P:X\to Y$ is a homogeneous polynomial of degree $d\geq 0$ if $P(x) = L(x,\cdots,x)$, where $L$ is a linear symmetric map from $\underbrace{X\times \cdots \times X}_{d\mbox{ times }}$ to $Y$ (for $d=0$, these are constant functions on $X$). A function $P:X\to Y$ is a polynomial of degree $d$ is $P = \sum_{j=0}^d P_j$, where $P_j$ is a homogeneous polynomial of degree $j$. 
We let $\wp_d(X,Y)$  be the set of all polynomials $P:\;X\to Y$ with degree at most $d$. When $Y=\rset$, we write $\wp_d(X)$. If $f:X\to Y$ has Frechet derivatives to order $k$ we write $\nabla^{(k)} f(x)$ to denote its $k$-th order Frechet derivative, and for $z\in X$, $\nabla^{(k)} f(x)\cdot z^k\eqdef \nabla^{(k)} f(x)(z,\ldots,z)$. Furthermore, we set
\[ \|\nabla^{(k)} f(x)\| \eqdef \sup_{z\in X:\;\|z\|_X\leq 1} \|\nabla^{(k)} f(x)(z,\ldots,z)\|_Y.\] 
Let $T_d(t)=\cos(d\arccos(t))$, $t\in[-1,1]$ denote the Chebyshev polynomial of degree $d$, and $T^{(j)}_d(t)$ its $j$-th order derivative. The following lemma is due to \cite{harris:08}~Theorem 1.

\begin{lemma}\label{lem:markov}
For $P\in\wp_d(\rset^m)$, and $k\geq 1$, let $P^{(k)}$ denote the $k$-th order derivative of $P$. We have
\[\sup_{x:\;\|x\|_2\leq 1} \|P^{(k)}(x)\| \leq T_d^{(k)}(1) \;\times \sup_{x:\;\|x\|_2\leq 1} |P(x)|,\]
where $\|P^{(k)}(x)\|\eqdef \sup_{y:\|y\|_2\leq 1} |P^{(k)}(x)(y,\ldots,y)|$.
\end{lemma}

We turn to the proof of Theorem \ref{thm:approx:deriv}. Let $\alpha=\bar\alpha$ as defined in the statement of the theorem, and fix $x_0\in\mathsf{int}_{\alpha}(\Omega)$. The function $u$ has derivatives to the order $\lfloor \beta\rfloor$ at $x_0$. By Taylor approximation we have, for all $z\in\rset^m$, with $\|z\|_2\leq 1$, setting $r\eqdef \lfloor \beta \rfloor -1$,
\begin{multline*}
u(x_0 + \alpha z) = \\
u(x_0) + \alpha\nabla u(x_0)\cdot z + \cdots + \frac{\alpha^r}{r!}\nabla^{(r)} u(x_0)\cdot z^r + \frac{\alpha^{\lfloor \beta \rfloor}}{r!}\int_0^1(1-t)^r \nabla^{(\lfloor \beta \rfloor)}u\left(x_0 +\alpha t z\right)\cdot z^{\lfloor \beta \rfloor}\rmd t\\
 = u(x_0) + \alpha \nabla u(x_0)\cdot z + \cdots + \frac{\alpha^{\lfloor \beta \rfloor}}{\lfloor \beta \rfloor !}\nabla^{(\lfloor \beta \rfloor)} u(x_0)\cdot z^{\lfloor \beta \rfloor} + R_u(z),
\end{multline*}
where, using the fact that $u\in\Cset^\beta(\Omega,M)$, the remainder satisfies
\begin{multline*}
|R_u(z)| = \frac{\alpha^{\lfloor \beta \rfloor}}{r!}\left|\int_0^1(1-t)^r \left(\nabla^{(\lfloor \beta \rfloor)}u\left(x_0 +\alpha t z\right) - \nabla^{(\lfloor \beta \rfloor)}u(x_0)\right)\cdot z^{\lfloor \beta \rfloor}\rmd t\right| \\
\leq \frac{\alpha^{\lfloor \beta \rfloor}}{r!}\int_0^1 (1-t)^r M (\alpha t)^{\beta- \lfloor \beta \rfloor}\|z\|_2^{\beta- \lfloor \beta \rfloor} \|z\|_\infty^{\lfloor \beta \rfloor}\rmd t \leq \frac{M \alpha^\beta}{\lfloor \beta \rfloor !}.
\end{multline*}
A similar expansion holds for $\tilde u$. And since $x_0+\alpha z\in\Omega$, we have $|u(x_0+\alpha z) -\tilde u(x_0+\alpha z)| \leq \varepsilon$. Therefore, setting
\begin{multline*}
P(z) \eqdef u(x_0)-\tilde u(x_0) + \alpha (\nabla u(x_0)-\nabla\tilde u(x_0))\cdot z \\
+ \cdots + \frac{\alpha^{\lfloor \beta \rfloor}}{\lfloor \beta \rfloor !}\left(\nabla^{(\lfloor \beta \rfloor)} u(x_0) - \nabla^{(\lfloor \beta \rfloor)} \tilde u(x_0)\right)\cdot z^{\lfloor \beta \rfloor},\end{multline*}
we conclude that
\begin{equation}\label{eq1}
\sup_{z\in\rset^m:\;\|z\|_2\leq 1} |P(z)| \leq \varepsilon + \frac{2M \alpha^{\beta}}{ \lfloor \beta \rfloor!}
\end{equation}
We note that the function $z\mapsto P(z)$ is infinitely differentiable polynomial of degree $\lfloor \beta \rfloor$ on $\rset^m$, and for $0\leq \tau\leq \lfloor \beta \rfloor$, and $y\in\rset^m$,
\begin{equation}\label{nabla:P}
\nabla^{(\tau)} P(z)\cdot y^\tau  = \sum_{j=\tau}^{\lfloor \beta \rfloor} \frac{\alpha^j}{(j-\tau)!} \left(\nabla u^{(j)}(x_0) - \nabla \tilde u^{(j)}(x_0)\right)\cdot\left(z^{j-\tau},y^\tau\right).\end{equation}

On the other hand, the function $x\mapsto \nabla^{(\tau)} u(x):\;\Omega\mapsto \mathcal{L}(\rset^m\times \cdots\times \rset^m;\rset)$ has derivatives to order $\lfloor \beta \rfloor -\tau$ and $(\nabla^{(j)} \nabla^{(\tau)} u(x_0)\cdot y^j)\cdot x^\tau = \nabla^{(\tau+j)}u(x_0)\cdot(y^j,x^\tau)$. Setting $k = \lfloor \beta \rfloor - \tau$, its Taylor expansion at $x_0$ yields
\begin{multline*}
\nabla^{(\tau)} u(x_0 +\alpha z) = \underbrace{\nabla^{(\tau)} u(x_0) + \alpha\nabla^{(\tau+1)}u(x_0)\cdot z + \cdots + \frac{\alpha^k}{k!}\nabla^{(\lfloor \beta \rfloor)} u(x_0)\cdot z^{k}}_{Q(z)} + R'(z),
\end{multline*}
and
\begin{multline*} 
\nabla^{(\tau)} \tilde u(x_0+\alpha z) = \underbrace{\nabla^{(\tau)} \tilde u(x_0) + \alpha \nabla^{(s+1)}\tilde u(x_0)\cdot z + \cdots + \frac{\alpha^k}{k!}\nabla^{(\lfloor \beta \rfloor)} \tilde u(x_0)\cdot z^{k}}_{\tilde Q(z)}+ \tilde R'(z).
\end{multline*}
Using the same calculations as above, we check that the remainders $R'$ and $\tilde R'$ satisfy
\[\sup_{z\in\rset^d:\;\|z\|_2\leq 1}\; \|R'(z)\| + \|\tilde R'(z)\| \leq \frac{2M\alpha^{\beta - \tau}}{(\lfloor \beta \rfloor - \tau)!}.\]
Using the last display, and noting that $\alpha^\tau(Q(z) -\tilde Q(z)) = \nabla^{(\tau)} P(z)$ given in (\ref{nabla:P}), it follows that for all $z$ in the unit ball,
\[\|\nabla^{(\tau)} u(x_0+\alpha z) - \nabla^{(\tau)} \tilde u(x_0+\alpha z)\| \leq \frac{1}{\alpha^\tau}\|\nabla^{(\tau)} P(z)\| +\frac{2M\alpha^{\beta - \tau}}{(\lfloor \beta \rfloor - \tau)!}.\]
Therefore, by Markov's polynomial inequality (Lemma \ref{lem:markov}), for all $z$ in the unit ball
\begin{multline*}
\|\nabla^{(\tau)} u(x_0+\alpha z) - \nabla^{(\tau)} \tilde u(x_0+\alpha z)\| \leq \frac{T^{(\tau)}_{\lfloor \beta \rfloor}(1)}{\alpha^\tau}\left(\varepsilon + \frac{2M\alpha^\beta}{\lfloor \beta \rfloor!}\right)  + \frac{2M \alpha^{\beta-\tau}}{(\lfloor \beta \rfloor -\tau)!}\\
= \frac{A}{\alpha^\tau} + B \alpha^{\beta-\tau},\end{multline*}
where
\[A = \varepsilon T^{(\tau)}_{\lfloor \beta \rfloor}(1),\;\; B = \frac{2M \times T^{(\tau)}_{\lfloor \beta \rfloor}(1)}{\lfloor \beta \rfloor!} + \frac{2M}{(\lfloor \beta \rfloor-\tau)!}.\]

With the choice
\[\alpha = \left(\frac{\tau A}{(\beta-\tau)B}\right)^{1/\beta} \leq \left(\frac{\tau\varepsilon \lfloor \beta \rfloor!}{2M(\beta-\tau)}\right)^{1/\beta},\]
we get
\[ \|\nabla^{(\tau)} u(x_0) - \nabla^{(\tau)} \tilde u(x_0)\|\leq C M^{\frac{\tau}{\beta}} \varepsilon^{\frac{\beta-\tau}{\beta}},\]
for some constant $C$ that depends only on $\tau$ and $\beta$. In fact, $C$ can be taken as
\[C = \frac{\beta T^{(\tau)}_{\lfloor \beta \rfloor}(1)}{\beta-\tau}\left(\frac{\beta-\tau}{\tau}\right)^{\tau/\beta} \leq \frac{\beta}{\tau} T^{(\tau)}_{\lfloor \beta \rfloor}(1) .\]
\end{proof}

\subsection{Proof of Theorem \ref{thm:main_param}}\label{sec:proof:thm:main_param}
We start with a lemma that establishes that the covering number of $\Gamma(\Theta)$ grows polynomially near $0$, which consequently implies that $\Gamma(\Theta)$ has a finite VC dimension. This effectively proves that $\Gamma(\Theta)$ is a finite-dimensional subset of $\cC^{\beta}(\Omega)$, serving as a cornerstone of our main theorem.

\begin{lemma}
\label{lem:cover_gamma}
Given any $\eps > 0$, we have: 
$$
\cN\left(\eps, \Gamma(\Theta), L_\infty\right) \le C\left(\frac{c}{\eps}\right)^d \,,
$$
for some constants $C, c > 0$. 
\end{lemma}
\begin{proof}
Recall that $\Theta$ is a compact subset of $\reals^d$, which implies $\vol(\Theta) < \infty$. Therefore, from the standard covering number calculation, we have: 
$$
\cN(\eps, \Theta, L_2) \le \frac{\vol(\Theta)}{\vol(B)}\left(\frac{3}{\eps}\right)^d \triangleq C\left(\frac{3}{\eps}\right)^d  \,,
$$
where $B$ is the unit ball (centered at the origin) in $\reals^d$ (e.g., see Theorem 14.2 of Yihong Wu's lecture notes, ECE598, Spring 2016). We next claim that: 
$$
\cN\left(\eps, \Gamma(\Theta), L_\infty\right) \le \cN(\eps/L, \Theta, L_2) \,.
$$
which will complete the proof. To show this, suppose $\{\theta_1, \dots, \theta_n\}$ are the center of $(\eps/L)$-covers of $\Theta$, i.e. 
$$
\sup_{\theta \in \Theta} \min_{1 \le j \le n} \|\theta_j - \theta\|_2 \le \eps/L \,.
$$
where $n = \cN(\eps/L, \Theta, L_2)$. Now consider the images corresponding to these $\theta$'s, meaning, $\{\gamma_1, \dots, \gamma_n\}$ where $\gamma_i = \Gamma(\theta_i)$. Then we have: 
$$
\sup_{\theta \in \Theta} \min_{1 \le j \le n} \|\Gamma(\theta_j) - \Gamma(\theta)\|_\infty \le L \sup_{\theta \in \Theta} \min_{1 \le j \le n}\|\theta_j - \theta\|_2 \le L (\eps/L) = \eps \,.
$$
Hence, the proof of the lemma is complete with $c = 3L$. 
\end{proof}

\begin{proof}[Proof of Theorem \ref{thm:main_param}]
We use the standard rate theorem (e.g., see 3.2.5 of \cite{vaart1997weak}) to establish the convergence rate of $|\hat u - u_\star|_2^2$. Define, for simplicity, $M(u) = \bbE[(Y - u({\bf s}))^2]$ and $M_n(u) = (1/n) \sum_i (Y_i - u({\bf s}_i))^2$. Therefore, from \eqref{eq:est} we have $\hat u$ is the minimizer of $M_n(u)$ over $\Gamma(\Theta)$. Furthermore, we have: 
$$
M(u) - M(u_\star) = |u - u_\star|_2^2 \,.
$$
To establish the rate of convergence, we need to find the modulus of continuity function $\phi_n(\delta)$ that satisfies: 
$$
\bbE\left[\sup_{|u - u_\star|_2 \le \delta}\left|\left(M_n(u) - M(u)\right) - \left(M_n(u_\star) - M(u_\star)\right)\right|\right] \le \frac{\phi_n(\delta)}{\sqrt{n}} \,.
$$
Let $\xi_i \eqdef(Y_i -u_\star({\bf s}_i))/\sigma$. Some simple algebra yields: 
\begin{align*}
& \left|\left(M_n(u) - M(u)\right) - \left(M_n(u_\star) - M(u_\star)\right)\right| \\
& = \left|(\bbP_n - \bbP)(u({\bf s}) - u_\star({\bf s}))^2 + \sigma(\bbP_n  - \bbP)\xi(u({\bf s}) - u_\star({\bf s}))\right| \\
& \le \left|(\bbP_n - \bbP)(u({\bf s}) - u_\star({\bf s}))^2\right| + \sigma \left|(\bbP_n  - \bbP)\xi(u({\bf s}) - u_\star({\bf s}))\right|
\end{align*}
Therefore, by symmetrization we have (using $(\eta_1, \dots, \eta_n)$ i.i.d Rademacher random variables): 
\begin{align*}
& \bbE\left[\sup_{|u - u_\star|_2 \le \delta}\left|\left(M_n(u) - M(u)\right) - \left(M_n(u_\star) - M(u_\star)\right)\right|\right]  \\
& \le \bbE\left[\sup_{|u - u_\star|_2 \le \delta} \left|(\bbP_n - \bbP)(u({\bf s}) - u_\star({\bf s}))^2\right|\right] + \sigma\bbE\left[\sup_{|u - u_\star|_2 \le \delta} \left|(\bbP_n  - \bbP)\xi(u({\bf s}) - u_\star({\bf s}))\right|\right] \\
& \le 2\bbE\left[\sup_{|u - u_\star|_2 \le \delta} \left|\frac1n \sum_i \eta_i (u({\bf s}_i) - u_\star({\bf s}_i))^2\right|\right] + \sigma\bbE\left[\sup_{|u - u_\star|_2 \le \delta} \left|\frac1n \sum_i \xi_i (u({\bf s}_i) - u_\star({\bf s}_i))\right|\right] \\
& \le 2B \bbE\left[\sup_{|u - u_\star|_2 \le \delta} \left|\frac1n \sum_i \eta_i (u({\bf s}_i) - u_\star({\bf s}_i))\right|\right] + \sigma\bbE\left[\sup_{|u - u_\star|_2 \le \delta} \left|\frac1n \sum_i \xi_i (u({\bf s}_i) - u_\star({\bf s}_i))\right|\right] \\
& \le C_1 \left(\delta \sqrt{\frac{d}{n}\log{\left(\frac{C_2}{\delta}\right)}} + \frac{d}{n} \log{\left(\frac{C_2}{\delta}\right)}\right)
\end{align*}
Here, in the second last inequality, we use Leduox-Talagrand contraction inequality. The last line follows Dudley's chaining bound and some standard calculations (e.g., proof of Theorem 8.7 of \cite{sen2018gentle}). Therefore, a valid choice of $\phi_n(\delta)$ is: 
$$
\phi_n(\delta) = C_1  \left(\delta \sqrt{d\log{\left(\frac{C_2}{\delta}\right)}} + \frac{d}{\sqrt{n}} \log{\left(\frac{C_2}{\delta}\right)}\right) \,.
$$
Therefore, we conclude that $\|\hat u - u_\star\|_2 = O_p(\delta_n)$ where $\delta_n$ satisfies: 
$$
\sqrt{n}\delta_n^2 \ge \phi_n(\delta_n) \,.
$$
Some simple algebra yields that a valid choice for $\delta_n$ is $\sqrt{(d/n)\log{(n/d)}}$, which completes the proof. 
\end{proof}

\medskip

\bibliographystyle{ims}
\bibliography{biblio}

\end{document}